\documentclass[12pt, a4paper]{article}
\usepackage[latin1]{inputenc}
\usepackage[english]{babel}
\usepackage{amsmath,verbatim,amsthm,mathrsfs,stmaryrd}
\usepackage[english]{babel}
\usepackage{hyperref}
\usepackage{indentfirst}
\usepackage{amstext}
\usepackage{enumerate}
\usepackage[all]{xy}
\usepackage{amsfonts} 
\usepackage{textcomp}
\usepackage{amssymb}
\usepackage[dvips]{graphicx}
\usepackage{setspace}
\usepackage{color}
\usepackage{ulem}
\usepackage[dvipsnames]{xcolor}
\usepackage{graphicx}
\graphicspath{ {./images/} }
\usepackage[all]{xy}
\usepackage[noabbrev, capitalise, nameinlink, nosort]{cleveref}

\newcommand{\field}[1]{\mathbb{#1}}

\newcommand {\R}{\mathbb{R}}
\newcommand {\N}{\mathbb{N}}
\newcommand{\GG}{\mathcal{G}}

\newcommand {\p}{\mathfrak{p}}

\newcommand {\Z}{\mathbb{Z}}

\newcommand{\NN}{\field{N}}

\newcommand{\bgln}{\begin{eqnarray}} 
\newcommand{\egln}{\end{eqnarray}}
\newcommand{\bgl}{\begin{equation}} 
\newcommand{\egl}{\end{equation}}

\usepackage{geometry}
\newcommand{\ro}{\rho}

\newtheorem{teorema}{theorem}[section]
\newtheorem{lemma}[teorema]{Lemma}
\newtheorem{corollary}[teorema]{Corollary}
\newtheorem{definition}[teorema]{Definition}
\newtheorem{proposition}[teorema]{Proposition}
\newtheorem{hypothesis}[teorema]{Hypothesis}
\newtheorem{example}[teorema]{Example}
\theoremstyle{remark}
\theoremstyle{definition}
\newtheorem{remark}[teorema]{Remark}
\newtheorem{theorem}[teorema]{Theorem}

\begin{document}
\onehalfspace

\author{Daniel Gon\c{c}alves\footnote{Partially supported by Conselho Nacional de Desenvolvimento Cient\'ifico e Tecnol\'ogico (CNPq) and Capes-PrInt - Brazil.}, Danilo Royer, and Felipe Augusto Tasca}

\title{Entropy of local homeomorphisms with applications to infinite alphabet shift spaces}
\maketitle

\begin{abstract} 

In this paper, we introduce topological entropy for dynamical systems generated by a single local homeomorphism (Deaconu-Renault systems). More precisely, we generalize Adler, Konheim, and McAndrew's definition of entropy via covers and Bowen's definition of entropy via separated sets. We propose a definition of factor map between Deaconu-Renault systems and show that entropy (via separated sets) always decreases under uniformly continuous factor maps. Since the variational principle does not hold in the full generality of our setting, we show that the proposed entropy via covers is a lower bound to the proposed entropy via separated sets. Finally, we compute entropy for infinite graphs (and ultragraphs) and compare it with the entropy of infinite graphs defined by Gurevich.

\vspace{5mm}


\vspace{5mm}

\end{abstract}

\section{Introduction}

Entropy is a classical concept in Mathematics,  with an ample spectrum of applications and many ramifications. For instance, in the context of C*-algebraic dynamics entropy has been studied by Kerr and Voicolescu, see \cite{Kerr, Kerr1, Voicolescu}, and more specifically, for graph and Exel-Laca C*-algebras, noncommutative entropy has been studied in \cite{JP06, JP09, MNTY}. 
 In our work, we are interested in the topological entropy of singly generated dynamical systems (also known as Deaconu-Renault systems). For these systems, we generalize Adler, Konheim, and McAndrew's definition of entropy via covers and Bowen's definition of entropy via separated sets, and start the study of the theory of these generalized entropies. 

In more precise terms, a Deaconu-Renault system \cite{CRST, D95, Re00} consists of a locally compact Hausdorff space and a local homeomorphism between open subsets of that space. Among several examples where such systems are found, we mention self-covering maps \cite{D95}, one-sided shifts of finite type \cite{Kitchens, LindMarcus}, the boundary-path space of a directed graph together with the shift map \cite{BCW, Webster}, the one-sided edge shift space of an ultragraph together with the restriction of the shift map to points with non-zero length \cite{GRultrapartial, TascaDaniel, GDW}, and the systems associated with algebras of one-sided subshifts over arbitrary alphabets \cite{BCGW}. 

Deaconu-Renault systems routinely appear as the key ingredient in the description of C*-algebras as groupoid C*-algebras (and of algebras as Steinberg algebras) and their dynamical properties are often linked with properties of the associated algebras. For example, conjugacy of Deaconu-Renault systems is described in terms of isomorphisms of the associated C*-algebras in \cite{ABCE}. In our work, we describe, among other things, when our proposed metric entropy is an invariant for conjugacy of Deaconu-Renault systems. 

The usual definition of topological entropy (either via covers or via Bowen's metric definition) is associated with a function from a compact space $X$ to itself (in Bowen's case, the space can be non-compact). Given a Deaconu-Renault system $(X,\sigma)$, we deal with the facts that the domain of the map $\sigma$ is only an open subset of $X$ (not necessarily the whole space) and the space $X$ is not necessarily compact. We use approximations of $X$ by compact sets (not necessarily invariant compact subsets, as done in \cite{MR2504577}) and approximate the domains of the powers of $\sigma$ by closed sets from inside (in the case of Bowen's metric entropy). When $(X,\sigma)$ is a dynamical system in the classical sense, our definitions coincide with the usual ones but, for the general case, all the theory has to be developed. For example, the classical variational principle, which guarantees that Bowen's entropy and Adler, Konheim, and McAndrew entropy coincide does not hold (this is expected, as in \cite{HK} it is proved that Bowen's entropy for non-compact sets depends on the metric). In fact, we show that the topological entropy (our generalization of Adler, Konheim, and McAndrew entropy) of a Deaconu-Renault system is a lower bound to the metric entropy (our generalization of Bowen's entropy), see Theorem~\ref{zebra}. Due to the length of the paper, we left it for future work to explore if the topological entropy is equal to the infimum of the metric entropies, as it happens in \cite{Patrao, Patrao1}.

With our approach via Deaconu-Renault systems, we can define the entropy of the edge shift spaces associated with infinite graphs or ultragraphs (which include infinite matrices). In the more specific context of an infinite locally-finite graph, there are two well-established definitions of entropy: one due to Salama \cite{Sal1} and one due to Gurevich \cite{gurevich} (we refer the reader to \cite{LindMarcus} for an overview). Gurevich's entropy can be computed as a metric entropy (using Bowen's definition) and as the supremum of the entropies of finite subgraphs. We show that this entropy coincides with the metric entropy of the Deaconu-Renault system associated with the locally finite graph. We also show that the topological entropy (via covers) of the Deaconu-Renault system associated with a row-finite graph can be computed as the supremum of the entropies of finite graphs. As a direct corollary, we obtain that Gurevich entropy can be computed via covers and that our definition generalizes the Gurevich definition of entropy.

One of the interesting facts about the entropy of infinite graphs is that Salama \cite{Sal1} and Petersen \cite{Pet2} have given examples where entropy can increase via factors with respect to either Salama or Gurevich entropy. Motivated by this, we propose a definition of a factor map between Deaconu-Renault systems and show that the metric entropy decreases under uniform continuous factor maps. We also prove results regarding the metric entropy of invariant subsets, which allow us to compute the entropy of graphs formed by disconnected pieces. 

As one can see in \cite{Kitchens}, the entropy of certain countable state Markov shifts (associated with infinite matrices) has been studied in the literature, including the development of a Perron-Frobenius theory. Our approach to infinite matrices is to study the entropy of the Deaconu-Renault system corresponding to the ultragraph associated with the matrix (see \cite{Tomforde} for details). To give a glimpse of the deep combinatorics involved, we compute the topological entropy of the Renewal shift (which is an important example of a topologically mixing countable Markov shift and has been studied in \cite{Sarig, BEFR, MNTY} for instance).

This paper is organized as follows. In \S 2, we present preliminaries and establish some facts about conjugacy of Deaconu-Renault systems. In \S 3, we define the metric entropy of a Deaconu-Renault system, which is based on generalizations of Bowen's definitions of spanning and separated sets, and show that a uniformly continuous conjugacy, with uniformly continuous inverse, between two Deaconu-Renault systems preserve the metric entropy. We focus on factor maps in \S 4: we propose a definition of a factor map between Deaconu-Renault systems and show that the metric entropy always decreases under a uniformly continuous factor map (Theorem~\ref{factorentropy}). The behavior of the metric entropy of invariant subsets is our focus on the short \S 5. In \S 6, we define the topological entropy (via covers), which is more general than the metric entropy, since it can be defined in the absence of a metric. We prove that, in the case of a metric Deaconu-Renault system, the topological entropy is a lower bound to the metric entropy (Theorem \ref{zebra}). In \S 7, we focus on computing the entropy of Deaconu-Renault systems associated with graphs and ultragraphs. For row-finite graphs, we show that the topological entropy can be computed as the supremum of entropy of finite subgraphs (Theorem~\ref{sup dos subgrafos}). To finalize, we compute entropy for several examples, including finite graphs, locally finite graphs, row-finite graphs, disconnected graphs, and the ultragraph associated with the Renewal shift.

\section{Preliminaries}

Throughout this paper, we consider the set of the natural numbers as the set $\{1,2, \ldots \}$, which we denote by $\N$. 

Deaconu-Renault systems, or singly generated dynamical systems, were introduced independently by Deaconu in \cite{D95} and by Renault in \cite{Re00}. In this subsection, we recall some of its concepts and terminology, as introduced in \cite{ABCE}, which in turn follows \cite{CRST}.

\begin{definition} \cite{ABCE} A Deaconu-Renault system is a pair $(X,\sigma)$ consisting of a locally compact Hausdorff space $X$, and a local homeomorphism $\sigma: Dom(\sigma)\longrightarrow Im(\sigma)$ from an open set $Dom (\sigma)\subseteq X$ to an open set $Im(\sigma)\subseteq X$. Inductively define $Dom(\sigma^n):=\sigma^{-1}(Dom(\sigma^{n-1}))$, so each $\sigma^n:Dom(\sigma^n)\longrightarrow Im(\sigma^n)$ is a local homeomorphism and $\sigma^m\circ \sigma^n=\sigma^{m+n}$ on $Dom(\sigma^{m+n})$.
\end{definition}

In our work, we will prove that entropy is invariant under a class of conjugacies. Therefore, we recall the notion of conjugacy between Deaconu-Renault systems below.

\begin{definition}(\cite[Definition~2.1]{ABCE})\label{def of conjugacy}
We say that the Deaconu-Renault systems $(X,\sigma_X)$ and $(Y,\sigma_Y)$ are conjugate if there exists a homeomorphism $\phi:X\rightarrow Y$ such that $\sigma_Y(\phi(x)) = \phi(\sigma_X(x))$ and $\sigma_X(\phi^{-1}(y)) = \phi^{-1}(\sigma_Y(y))$ for all $x\in Dom(\sigma_X) $, and $y\in Dom(\sigma_Y)$. In this case, we say that $\phi$ is a conjugacy.
\end{definition}

Conjugacies between Deaconu-Renault systems can be further described as follows. Note that the conditions in \cref{item:inverse-images} below are set equalities.

\begin{lemma}\label{altconjug}\cite[Lemma~2.6]{ABCE} 
Let $(X,\sigma_X)$ and $(Y,\sigma_Y)$ be Deaconu--Renault systems, and let $\phi\colon X \to Y$ be a homeomorphism. The following are equivalent:
\begin{enumerate}
\item \label{item:conjugacy} $\phi\colon X \to Y$ is a conjugacy;
\item \label{item:domain} $\phi(Dom(\sigma_{X})) = Dom(\sigma_{Y})$, and $\phi \circ \sigma_X = \sigma_Y \circ \phi$ on $Dom(\sigma_{X})$;
\item \label{item:inverse-images} $\phi(\sigma_X^{-1}(x)) = \sigma_Y^{-1}(\phi(x))$ and $\phi^{-1}(\sigma_Y^{-1}(y)) = \sigma_X^{-1}(\phi^{-1}(y))$, for all $x \in X$ and $y \in Y$.
\end{enumerate}
\end{lemma}

Next, we show that a conjugacy preserves the domains of the iterates of the local homeomorphism $\sigma$.

\begin{lemma}\label{lemaconjuga}
Let $(X,\sigma_X)$ and $(Y,\sigma_Y)$ be Deaconu--Renault systems, and let $\phi\colon X \to Y$ be a conjugacy. Then, for each $i\in \N$, $\phi(Dom(\sigma_{X}^i)) = Dom(\sigma_{Y}^i)$, $\phi \circ \sigma_X^i = \sigma_Y^i \circ \phi$ on $Dom(\sigma_{X}^i)$, and $\sigma_X^i \circ \phi^{-1}= \phi^{-1} \circ \sigma_Y^i$ on $Dom(\sigma_{Y}^i)$. 
\end{lemma}
\begin{proof}
We prove the lemma by induction over $i$. The case $i=1$ follows from Definition $\ref{def of conjugacy}$. 
Let $i>1$ and suppose that $\phi(Dom(\sigma_{X}^k)) = Dom(\sigma_{Y}^k)$, $\phi \circ \sigma_X^k = \sigma_Y^k \circ \phi$ on $Dom(\sigma_{X}^k)$ and $\sigma_X^k \circ \phi^{-1}= \phi^{-1} \circ \sigma_Y^k$ on $Dom(\sigma_{Y}^k)$ for all $k\in \{1,\ldots , i-1\}$. 

If $x\in Dom(\sigma_X^i)$, then $\sigma_X^{i-1}(x)\in Dom(\sigma_X)$ and so $\phi(\sigma_X(\sigma_X^{i-1}(x)))=\sigma_Y(\phi(\sigma_X^{i-1}(x)))=\sigma_Y^i(\phi(x))$. This means that $\phi(Dom(\sigma_X^i))\subseteq Dom(\sigma_Y^i)$ and  $\phi \circ \sigma_X^i = \sigma_Y^i \circ \phi$. To show that $\phi(Dom(\sigma_X^i))\supseteq Dom(\sigma_Y^i)$, fix $y\in Dom(\sigma_Y^i)$. In this case, $\sigma_Y^{i-1}(y)\in Dom(\sigma_Y)=\phi(Dom(\sigma_X))$. Therefore, $\phi^{-1}(\sigma_Y^{i-1}(y))\in Dom(\sigma_X)$ and we can apply $\sigma_X$. So $\phi^{-1}(\sigma_Y^i(y))=\phi^{-1}(\sigma_Y(\sigma_Y^{i-1}(y)))=\sigma_X(\phi^{-1}(\sigma_Y^{i-1}(y)))=\sigma_X(\sigma_X^{i-1}(\phi^{-1}(y)))=\sigma_X^i(\phi^{-1}(y))$, which means that $\phi^{-1}(y)\in Dom(\sigma_X^i)$ and so $y\in \phi(Dom(\sigma_X^i))$. The above equality also shows that $\sigma_X^i \circ \phi^{-1}= \phi^{-1} \circ \sigma_Y^i$ on $Dom(\sigma_{Y}^i)$.
\end{proof}

\section{The metric entropy of Deaconu-Renault systems.}

In this section, we always assume that $(X,\sigma)$ is a Deaconu-Renault system with $X$ a metric space with metric $d$. 
  For $n\in\NN$ and $x\in X$, using the convention that $Dom (\sigma^0)=X$, we define 
\begin{align}\label{def:I_n}
 I_n(x)=\{i\in\{0,1,\dots, n-1\}:x\in Dom(\sigma^i)\}.
 \end{align}
 
 \begin{remark} Clearly, $I_n(x)\not=\emptyset$ for any $n\in \N$ and $x\in X$, since $0$ is always in $I_n(x)$.
 Moreover, $I_n(x)$ is always an initial segment of $\N\cup \{0\}$. More specifically, either $I_n(x)=\{0,\ldots ,n-1\}$, or 
 $I_n(x)=\{0,\ldots ,m\}$ with $m<n-1$, where $m$ is the greatest natural number $i$ such that $x\in Dom(\sigma^i)$ (which happens if $x\notin Dom(\sigma^{n-1})$).
 
  In the study of entropy of partial actions of $\Z$, the same set $I_n(x)$ as above is defined, see \cite{Entropyforhomeo}. Nevertheless, in the partial action setting it is not necessary that $I_n(x)$ is an initial segment of $\N\cup \{0\}$.
 \end{remark}
 
 To measure the distance between the iterates of two points, we define the map $d_n:X\times X\to[0, +\infty)$, for $n\in \N$, by
 $$
 d_n(x,y)=\max_{i\in I_n(x)\cap I_n(y)} d(\sigma^i(x),\sigma^i(y)).
 $$

Notice that if $I_n(x)\cap I_n(y)=\{0\}$  then $d_n(x,y)=d(x,y)$.
Another important observation is that, as it happened with partial actions (see \cite{Entropyforhomeo}), $d_n$ may not be a metric, because it is not possible to guarantee the triangular inequality, as the next example shows.

\begin{example}\label{contraexemplo}
Let $(X,\sigma)$ be the Deaconu-Renault system where $X=[0,1)$, $\sigma$ is given by $\sigma (x) = 2 x$, and $Dom(\sigma)=[0,\frac{1}{2})$. Inductively, we obtain that $Dom(\sigma^n)=[0,\frac{1}{2^n})$, $Im(\sigma^n)=[0,1)$ and $\sigma^n (x) = 2^n x$, for all $n\geq 0$. Notice that $d_3(0,\frac{6}{25})= \frac{24}{25}$. On the other hand, $d_3(0,\frac{1}{4})+d_3(\frac{1}{4},\frac{6}{25}) = \frac{1}{2}+\frac{1}{50}= \frac{13}{25}$, since $I_n(\frac{1}{4})=\{0,1\}, \forall n\geq 2$. Hence, the triangle inequality is not satisfied for $d_3$. 
\end{example}

Even if we require the domains of the local homeomorphisms to be clopen, $d_n$ may still fail the triangle inequality, as we show in the next example.

\begin{example}\label{siri} Let $X=\prod\limits_{n=0}^\infty \{0,1\}$ and, for $(x_n)_n, (y_n)_n\in X$, define $d((x_n)_n,(y_n)_n) = \frac{1}{2^i}$, where $i$ is the smallest index where $x_n\neq y_n$. We define a local homeomorphism on $X$: let $[0^n]=\{(x_j)_j\in X: x_j = 0 \text{ for each } \ 1\leq j \leq n\}$, $Dom(\sigma^n)=[0^{3n}]$, and $Im(\sigma^n)=X$, for every $n\in \N$. Of course, $Dom(\sigma^0)=X$. The homeomorphism $\sigma^n:[0^{3n}] \rightarrow X$, for $n\in \N\cup \{0\}$, is defined by $\sigma^n((x_i)_i)= (x_{i+3n})_i $. Let $x=000111\ldots $, $y=0000111\ldots $, and $z=00111\ldots $. Then, $d_2(x,y)=1$ but $d_2(x,z)= \frac{1}{4}$ and $d_2(z,y)= \frac{1}{4}$.

\end{example}

\begin{remark}\label{dn restricted} As seen in Examples \ref{contraexemplo} and \ref{siri}, the map $d_n$ may not be a metric, since the triangular inequality may fail. However, the restriction of $d_n$ to $Dom(\sigma^{n-1})$ is a metric, since the triangular inequality holds in $Dom(\sigma^{n-1})$.
\end{remark}

The following lemma will be used later.

\begin{lemma}\label{lema d_n} Let $(X,\sigma)$ be a Deaconu-Renault system, $\varepsilon>0$, $n\in \N$, and $x_1,x_2\in X$ be such that $d_n(x_1,x_2)>\varepsilon$. Then, there exists open neighborhoods $V_1$ of $x_1$ and $V_2$ of $x_2$ such that $d_n(z_1,z_2)>\varepsilon$ for all $z_1\in V_1$ and $z_2\in V_2$.
\end{lemma}
\begin{proof} Since $d_n(x_1,x_2)>\varepsilon$, then $d(\sigma^j(x_1),\sigma^j(x_2))>\varepsilon$ for some $j\in I_n(x_1)\cap I_n(x_2)$. 
Let $r=\frac{d(\sigma^j(x_1),\sigma^j(x_2))-\varepsilon}{2}$. Define $V_1=\sigma^{-j}(B(\sigma^j(x_1),r))$ and $V_2=\sigma^{-j}(B(\sigma^j(x_2),r))$, which are open neighborhoods of $x_1$ and $x_2$, respectively. Let $z_1\in V_1$ and $z_2\in V_2$. Then  $\sigma^j(z_i)\in B(\sigma^j(x_i),r)$ for $i=1,2$. Suppose that $d(\sigma^j(z_1), \sigma^j(z_2))\leq \varepsilon$. In this case, it follows that $d(\sigma^j(x_1), \sigma^j(x_2))\leq d(\sigma^j(x_1),\sigma^j(z_1))+d(\sigma^j(z_1),\sigma^j(z_2))+d(\sigma^j(z_2),\sigma^j(x_2))<2r+\varepsilon = d(\sigma^j(x_1),\sigma^j(x_2))$, which is a contradiction. So, it follows that $d(\sigma^j(z_1),\sigma^j(z_2))>\varepsilon$ and hence $d_n(z_1,z_2)>\varepsilon$ for all $z_1\in V_1$ and $z_2\in V_2$.
\end{proof}

To define the metric entropy of a Deaconu-Renault system, we need the notion of separated and spanning sets. Since the domain of the local homeomorphism does not need to be closed, we develop a generalization of these concepts below. 

\begin{definition}\label{definicaosepspan} Let $(X,\sigma)$ be a Deaconu-Renault system, $K$ a compact subset of $X$, $n\in\N$, and $Y$ a subset of $X$. Given $\varepsilon>0$, we say that:
\begin{itemize}
 \item a subset $A\subset K\cap Y$ is an $(n,\varepsilon,\sigma, K\cap Y)$-separated set if $d_n(x,y)>\varepsilon$ for every $x,y\in A $, $x\neq y$.
 
 \item a subset $B\subset K\cap Y$ is an $(n,\varepsilon,\sigma, K\cap Y)$-spanning set if, for every $x\in K\cap Y$, 
 there exists $y\in B$ such that  $d_n(x,y)\leq\varepsilon$.
 \end{itemize}
 We use the notation $sep(n,\varepsilon, \sigma, K\cap Y)$ to denote the largest cardinality of the $(n,\varepsilon,\sigma, K\cap Y)$-separated sets and $span(n,\varepsilon, \sigma, K\cap Y)$ to denote the smallest cardinality of the $(n,\varepsilon,\sigma, K\cap Y)$-spanning sets. 
 \end{definition}

 Closed subsets $F$ inside the domain of $\sigma^{n-1}$ will play a key role in the definition of spanning and separated sets of compact sets. Next, we describe how $sep(n,\varepsilon, \sigma, K\cap F)$ and $span(n,\varepsilon, \sigma, K\cap F)$ relate and show that these quantities are always finite. 

\begin{lemma}\label{sirigordo}
Let $(X,\sigma)$ be a Deaconu-Renault system, $K$ a compact subset of $X$, $n\in\N$, and $F$ a closed set of $X$ contained in $Dom(\sigma^{n-1})$.
Then, for each $\varepsilon>0$ and $n\in \N$ we have that:
\begin{enumerate}
    \item There exist a finite $(n,\varepsilon,\sigma, K\cap F)$-spanning set. Hence, $span(n,\varepsilon, \sigma, K\cap F)$ is always finite.
    \item The following inequalities are true:
\[ span(n,\varepsilon, \sigma, K\cap F)\leq sep(n,\varepsilon, \sigma, K\cap F) \leq span(n,\frac{\varepsilon}{2}, \sigma, K\cap F).
\]
\end{enumerate}
\end{lemma}
\begin{proof} Let $X,K,F$ be as in the hypotheses and fix $\varepsilon>0$ and $n\in \N$.

For each $x\in K\cap F$ consider the open set $$U(x,n,\varepsilon)=\displaystyle\bigcap_{i\in I_n(x)}\sigma^{-i}\left(\mathcal{B}(\sigma^i(x),\varepsilon)\right).$$ Notice that $d_n(x,y)<\varepsilon$ for each $y\in U(x,n,\varepsilon)$. Moreover, the set $U:=\bigcup\limits_{x\in K\cap F} U(x,n,\varepsilon)$ is an open cover of $K\cap F$ and, since $K\cap F$ is compact, there is a finite subcover 
 $$U(x_1,n,\varepsilon)\cup \ldots \cup U(x_m,n,\varepsilon).$$ Therefore, the set $B=\{x_1,\ldots ,x_m\}$ is an $(n,\varepsilon,\sigma, K\cap F)$-spanning set.
 
For the second item, we prove first that 
$sep(n,\varepsilon, \sigma, K\cap F) \leq span(n,\frac{\varepsilon}{2}, \sigma, K\cap F)$. 
Let $A$ be an $(n,\varepsilon, \sigma, K\cap F)$-separated set and $B$ an $(n,\frac{\varepsilon}{2}, \sigma, K\cap F)$-spanning set. Define a map $\Phi:A\rightarrow B$ by picking, for each point $x\in A$, a point $\Phi(x)$  in $B$ such that $d_n(x,\Phi(x))\leq \frac{\varepsilon}{2}$. Notice that $\Phi$ is injective. Indeed, if $x,y\in A$ are such that $\Phi(x)=\Phi(y)$ then, since $d_n$ satisfies the triangle inequality in $F$ (see Remark \ref{dn restricted}), we have that $d_n(x,y)\leq d_n(x,\Phi(x))+d_n(\Phi(x),y)\leq \varepsilon$ and, since $A$ is $(n,\varepsilon, \sigma, K\cap F)$-separated, we conclude that $x=y$. From the injectivity of $\Phi$ we obtain that the cardinality of $A$ is less or equal to the cardinality of $B$ and hence $sep(n,\varepsilon, \sigma, K\cap F) \leq span(n,\frac{\varepsilon}{2}, \sigma, K\cap F)$.

Now, let $C$ be a $(n,\varepsilon,\sigma, K\cap F)$-separated set of maximum cardinality. Then, for each $x\in (K\cap F)\setminus C$, there exists some $y\in C$ such that $d_n(x,y)\leq \varepsilon$, since otherwise $C\cup\{x\}$ is an $(n,\varepsilon,\sigma,K\cap F)$-separated set with cardinality greater than the cardinality of $C$. Therefore, $C$ is an $(n,\varepsilon, \sigma, K\cap F)$-spanning set, from where the first inequality of Item~2 follows.

\end{proof}

\begin{remark}\label{dynamicball}
For $x\in X$, $n\in \N$ and $\varepsilon>0$, it is customary to call the open set $U(x,n,\varepsilon)=\displaystyle\bigcap_{i\in I_n(x)}\sigma^{-i}\left(\mathcal{B}(\sigma^i(x),\varepsilon)\right)$ an open dynamical ball at $x$, see for example {\cite{Entropyforhomeo}}. As we observed in the proof above, $d_n(x,y)<\varepsilon$ for each $y\in U(x,n,\varepsilon)$.
\end{remark}

In our approach to define the entropy of Deaconu-Renault systems, we approximate $Dom(\sigma^{n-1})$, which is open, by closed sets from inside. In this manner, we obtain a generalization of the usual notion of separated and spanning sets, which we make precise below.

\begin{definition}\label{papagaio} Let $(X,\sigma)$ be a Deaconu-Renault system, $K$ be a compact subset of $X$, $n\in\N$, and $\varepsilon>0$. Define
\[ssep(n,\varepsilon, \sigma, K)= \displaystyle \sup_{F}sep(n,\varepsilon, \sigma, K\cap F),\] and 
\[sspan(n,\varepsilon, \sigma, K)= \displaystyle \sup_{F}span(n,\varepsilon, \sigma, K\cap F),\]
 where the suprema are taken over all closed sets $F\subseteq Dom(\sigma^{n-1}).$
\end{definition}

\begin{remark}\label{uniao de fechados}
It is always possible to approximate an open set from the inside with closed sets. Indeed, let $X$ be a metric space and $A\subseteq X$ be an open set. Define, for each $i\in \N$, 
$B_i=\{x\in X:d(x,A^c)\geq \frac{1}{i}\}$. Notice that $B_i\subseteq B_{i+1}\subseteq A$ for each $i\in \N$ and  $\bigcup\limits_{i\in \N}B_i=A$. 
\end{remark}

The following result is a direct consequence of Lemma~\ref{sirigordo}.

\begin{proposition}\label{leao}
Let $(X,\sigma)$ be a Deaconu-Renault system, $K$ a compact subset of $X$, $n\in\N$, and $\varepsilon>0$. Then, 
\[ sspan(n,\varepsilon, \sigma, K)\leq ssep(n,\varepsilon, \sigma, K) \leq sspan(n,\frac{\varepsilon}{2}, \sigma, K).
\]
\end{proposition}

It is important to notice that $ ssep(n,\varepsilon, \sigma, K)$ can be defined directly, without approximating from inside. This is the content of the next proposition.

\begin{proposition}\label{propnovossepspan}
Let $(X,\sigma)$ be a Deaconu-Renault system, $K$ a compact subset of $X$, $n\in\N$, and $\varepsilon>0$. Then: 
\begin{enumerate}
    \item $$ ssep(n,\varepsilon, \sigma, K) = sep(n,\varepsilon, \sigma, K\cap Dom(\sigma^{n-1})).$$
    \item $$span(n,\varepsilon, \sigma, K\cap Dom(\sigma^{n-1})) \leq sspan(n,\frac{\varepsilon}{2}, \sigma, K).$$
    \end{enumerate}
\end{proposition}

\begin{proof}
We begin with the first item. Let $F\subseteq Dom(\sigma^{n-1})$ be closed in $X$ and $B$ an $(n, \varepsilon, \sigma, K\cap F)$-separated set. Then, $B$ is also an $(n, \varepsilon, \sigma, K\cap Dom(\sigma^{n-1}))$-separated set and therefore $sep(n,\varepsilon, \sigma, K\cap F)\leq sep(n,\varepsilon, \sigma, K\cap Dom(\sigma^{n-1}))$. Since this inequality holds for each closed set $F$ in $X$ with $F\subseteq Dom(\sigma^{n-1})$, we conclude that $ssep(n,\varepsilon, \sigma, K)\leq sep(n,\varepsilon, \sigma, K\cap Dom(\sigma^{n-1}))$.

It remains to prove the reverse inequality. Let $A\subseteq K\cap Dom(\sigma^{n-1})$ be an  $(n,\varepsilon,\sigma, K\cap Dom(\sigma^{n-1}))$-separated set of largest cardinality. If $A$ is a finite set, then $A$ is closed in $X$ 
and $sep(n,\varepsilon, \sigma, K\cap A)$ is equal to the cardinality of $A$. Since $A\subseteq Dom(\sigma^{n-1})$, we obtain that $sep(n,\varepsilon, \sigma, K\cap Dom(\sigma^{n-1}))\leq ssep(n,\varepsilon, \sigma, K)$. If $A$ is an infinite set then, for each finite subset $B$ of $A$, we have that $B$ is closed in $X$, $B\subseteq K\cap Dom(\sigma^{n-1})$, and $sep(n, \varepsilon, \sigma, K\cap B)$ is equal to the cardinality of $B$. Therefore, $ssep(n, \varepsilon, \sigma, K)$ is also infinite and hence $sep(n,\varepsilon, \sigma, K\cap Dom(\sigma^{n-1}))= ssep(n,\varepsilon, \sigma, K)$ as desired.

Next, we prove the second item. 

Clearly, the inequality holds if $sspan(n,\frac{\varepsilon}{2},\sigma, K)=\infty$. Suppose that $sspan(n,\frac{\varepsilon}{2},\sigma, K)<\infty$. Let  $F\subseteq Dom(\sigma^{n-1})$ be a closed set such that $span(n,\frac{\varepsilon}{2},\sigma, K\cap F)=sspan(n,\frac{\varepsilon}{2},\sigma, K)$ and $B$ an $(n,\frac{\varepsilon}{2},\sigma, K\cap F)$ spanning set of minimal cardinality. We show below that $B$ is an $(n,\varepsilon, \sigma,  K\cap Dom(\sigma^{n-1}))$-spanning set.

Suppose that there exists an $y\in K\cap Dom(\sigma^{n-1})$ such that $d_n(y,b)>\varepsilon$ for each $b\in B$. Then $y\notin F$, since $B$ is an $(n,\frac{\varepsilon}{2},\sigma, K\cap F)$ spanning set. Moreover, for each $x\in K\cap F$ we have that  $d_n(x,y)>\frac{\varepsilon}{2}$. Indeed, suppose that there exists $x\in K\cap F$ such that $d_n(x,y)\leq \frac{\varepsilon}{2}$. Choose $b\in B$ such that $d_n(x,b)\leq \frac{\varepsilon}{2}$. Then, $d_n(y,b)\leq d_n(y,x)+d_n(x,b)\leq \varepsilon$, which is impossible, since $d_n(y,b)>\varepsilon$ for each $b\in B$. Now, define $F_1=F\cup \{y\}$, which is a closed set contained in $Dom(\sigma^{n-1})$. Let $C$ be an $(n,\frac{\varepsilon}{2}, \sigma, K\cap F_1)$-spanning set of minimal cardinality. Notice that $C=C_1\cup \{y\}$, where $C_1\subseteq K\cap F$ (since $d_n(x,y)>\frac{\varepsilon}{2}$ for each $x\in K\cap F$). Moreover, $C_1$ is an $(n,\frac{\varepsilon}{2},\sigma, K\cap F)$-spanning set.
Recall that $$span(n, \frac{\varepsilon}{2}, \sigma, K\cap F)=sspan(n, \frac{\varepsilon}{2}, \sigma, K)=$$ $$=sup\{span(n,\frac{\varepsilon}{2},\sigma, K\cap Y):Y\subseteq Dom(\sigma^{n-1}) \text{ is closed in }X\}.$$ Since $C$ is among the sets in which the supremum above is taken, we have that the cardinality of $C$ is less or equal to the cardinality of $B$, and since $C=C_1\cup\{y\}$ we conclude that the cardinality of $C_1$ is less than the cardinality of $B$. But this is a contradiction, since $B$ is an $(n,\frac{\varepsilon}{2},\sigma, K\cap F)$ spanning set of minimal cardinality.

We have proved above that $B$ is an $(n,\varepsilon, \sigma,  K\cap Dom(\sigma^{n-1}))$-spanning set. Hence, $$span(n,\varepsilon, \sigma, K\cap Dom(\sigma^{n-1}))\leq sspan (n,\frac{\varepsilon}{2}, \sigma, K),$$ as desired.

\end{proof}

For compact open sets, another possibility to define spanning and separating sets is to intersect with the closure of the domain of $\sigma^{n-1}$. 
This will be useful when we relate the entropy of this section with the entropy defined via covers. The details are below.

\begin{proposition}\label{compactopenprop} Let $(X,\sigma)$ be a Deaconu-Renault system and $K\subseteq X$ be a compact open set. Then, for each $\varepsilon>0$ and $n\in\N$, the following inequalities are true.
\begin{enumerate}
    \item $sep(n,\varepsilon, \sigma, K\cap Dom(\sigma^{n-1}))\leq sep(n,\varepsilon, \sigma, K\cap \overline{Dom(\sigma^{n-1})})\leq sep(n,\frac{\varepsilon}{2}, \sigma, K\cap Dom(\sigma^{n-1})).$
    \item $span(n,\varepsilon, \sigma, K\cap \overline{Dom(\sigma^{n-1})})\leq span(n,\varepsilon, \sigma, K\cap Dom(\sigma^{n-1})).$
\end{enumerate}
 
\end{proposition}

\begin{proof} The first inequality in Item~1. is straightforward, since each $(n,\varepsilon, \sigma, K\cap Dom(\sigma^{n-1}))$-separated set is also $(n,\varepsilon, \sigma, K\cap\overline{Dom(\sigma^{n-1})})$-separated.

 For the second inequality in Item~1., let $A\subseteq K\cap \overline{Dom(\sigma^{n-1})}$ be an  $(n,\varepsilon,\sigma, K\cap \overline{Dom(\sigma^{n-1})})$-separated set of largest cardinality. For each $a\in A$, the open set $ U(a,n,\frac{\varepsilon}{4})\cap K$ is a neighborhood of $a$ (where $ U(a,n,\frac{\varepsilon}{4})$ is a dynamic ball at $a$, see Remark~\ref{dynamicball}) and hence there exists $x_a \in Dom(\sigma^{n-1})\cap U(a,n,\frac{\varepsilon}{4})\cap K$.  Define $$B=\{x_a:a\in A\}.$$ 
 Recall that $d_n(a,x_a)<\frac{\varepsilon}{4}$ for all $a\in A$.
 We show that $d_n(x_a,x_b)> \frac{\varepsilon}{2}$ for each $x_a,x_b \in B$: let $a, b\in A$ and let $i\in I_n(a)\cap I_n(b)$ be such that $d(\sigma^i(a),\sigma^i(b))>\varepsilon$. Since $d_n(a,x_a)< \frac{\varepsilon}{4}$, we have that  $d(\sigma^i(a),\sigma^i(x_a))< \frac{\varepsilon}{4}$, and similarly $d(\sigma^i(b),\sigma^i(x_b))< \frac{\varepsilon}{4}$.
Then,
$$\varepsilon< d(\sigma^i(a),\sigma^i(b))\leq d(\sigma^i(a),\sigma^i(x_a))+d(\sigma^i(x_a),\sigma^i(x_b))+d(\sigma^i(x_b),\sigma^i(b))< $$
$$< d(\sigma^i(x_a),\sigma^i(x_b))+\frac{\varepsilon}{2}.$$
Therefore, $d(\sigma^i(x_a),\sigma^i(x_b))> \frac{\varepsilon}{2}$ and hence $d_n(x_a,x_b)> \frac{\varepsilon}{2}$.
Since $B\subseteq K\cap Dom(\sigma^{n-1})$ and the cardinality of $B$ is equal to the cardinality of $A$, we conclude that $sep(n,\varepsilon, \sigma, K\cap \overline{Dom(\sigma^{n-1})})\leq sep(n,\frac{\varepsilon}{2}, \sigma, K\cap Dom(\sigma^{n-1}))$, as desired.

Next, we prove Item 2. Clearly, the inequality holds if $span(n,\varepsilon,\sigma, K\cap Dom(\sigma^{n-1}))=\infty$. Suppose that $span(n,\varepsilon,\sigma, K\cap Dom(\sigma^{n-1}))<\infty$. Let $B$ be an $(n,\varepsilon,\sigma, K\cap Dom(\sigma^{n-1}))$-spanning set of minimal cardinality. We show that $B$ is also an  $(n,\varepsilon,\sigma, K\cap \overline{Dom(\sigma^{n-1})})$-spanning set: let $x\in K\cap \overline{Dom(\sigma^{n-1})}$. For each $m\in \N$, let $x_m\in K\cap Dom(\sigma^n)$ be such that $d_n(x,x_m)\leq \frac{1}{m}$ and  $b_m\in B$ be such that $d_n(x_m,b_m)\leq \varepsilon$. Since $B$ is a finite set, there exists $b \in B$ such that $b_m=b$ for infinitely many $m\in \N$. Hence, there is a subsequence $(m_j)_{j\in \N}$ such that $b_{m_j}=b$ for each $j\in \N$. 
Let $i\in I_n(x)\cap I_n(b)$ be such that $d_n(x,b)=d(\sigma^i(x),\sigma^i(b))$. Then, for each $j\in \N$, we have that 
$$d_n(x,b)=d(\sigma^i(x),\sigma^i(b))\leq d(\sigma^i(x),\sigma^i(x_{m_j}))+d(\sigma^i(x_{m_j}),\sigma^i(b))\leq $$
$$\leq d_n(x,x_{m_j})+d_n(x_{m_j},b_{m_j})\leq \frac{1}{m_j}+\varepsilon.$$
Therefore, $d_n(x,b)\leq \varepsilon$. This shows that $B$ is an $(n,\varepsilon,\sigma, K\cap \overline{Dom(\sigma^{n-1})})$-spanning set, and so $$span(n,\varepsilon, \sigma, K\cap \overline{Dom(\sigma^{n-1})}) \leq span(n,\varepsilon, \sigma, K\cap Dom(\sigma^{n-1})).$$
\end{proof}

We now define the metric entropy of a Deaconu-Renault system, in a way that it coincides with the usual definition of metric entropy when the domain of the homeomorphism $\sigma$ is the whole space $X$.

\begin{definition}\label{defentropy}Let $(X,\sigma)$ be a Deaconu-Renault system on a metric space $(X,d)$.
\begin{enumerate}
    
\item For each compact set $K\subseteq X$ and $\varepsilon>0$ define
$$
h_{\varepsilon}(\sigma,K,d)=\limsup_{n\to\infty}\frac{1}{n}\log \mbox{ssep}(n,\varepsilon,\sigma,K).$$
\item For each compact set $K\subseteq X$ define
$$
h_{d}(\sigma,K)=\lim_{\varepsilon\to0}h_{\varepsilon}(\sigma,K,d).$$
\item 
Define the metric entropy of the Deaconu-Renault system $(X,\sigma)$ as

\begin{equation}
h_d(\sigma)=\sup_{K\subseteq X}h_{d}(\sigma,K)
\end{equation}
where the supremum is taken over all the compact sets $K\subseteq X$.

\end{enumerate}
\end{definition}

\begin{remark} Recall that the limsup is well-defined for sequences of extended reals (where the extended reals $\overline{\R}=\R \cup \{-\infty,\infty\}$ are equipped with the usual order $-\infty<r<\infty$ for all $r\in \mathbb R$ and the order topology).  
In Item 1. above, there may exist some $n$ such that $ssep (n,\varepsilon, \sigma, K)=\infty$, in which case we set $\frac{1}{n}\log \mbox{ssep}(n,\varepsilon,\sigma,K)$ as $\infty$ and consider the limsup of a sequence of extended reals.

In Item~2. above, the quantity $\mbox{ssep}(n,\varepsilon,\sigma,K)$ increases monotonically as $\varepsilon$ decreases, and so $h_{\varepsilon}(\sigma,K,d)$ does as well. Therefore, the limit as $\varepsilon\rightarrow 0$  does exist. Furthermore, the inequalities in Proposition~\ref{leao}
imply that equivalent definitions of $h_d(\sigma)$ can be obtained
 if we replace $\mbox{ssep}(n, \varepsilon, \alpha,K)$ by
 $\mbox{sspan}(n,\varepsilon,\alpha,K)$ in the first item of Definition \ref{defentropy}.

\end{remark}

\begin{remark} As a consequence of the first item  of Proposition \ref{propnovossepspan}, we can use $sep(n, \varepsilon, \sigma, K\cap Dom(\sigma^{n-1}))$, instead of $ssep(n, \varepsilon, \sigma, K)$ in the definition of entropy above. Moreover, if $(X,\sigma)$ is a Deaconu-Renault system with a basis of compact open sets then, by Proposition~\ref{compactopenprop} and Proposition~\ref{propnovossepspan}, we can replace $ssep(n,\varepsilon,\sigma, K)$ with $sep(n, \varepsilon, \sigma, K\cap \overline{Dom(\sigma^{n-1})})$ in the first item of Definition~\ref{defentropy} and, to obtain the metric entropy $h_d(\sigma)$ we can take supremum over all the compact open sets of $X$, instead of the supremum over all the compact sets (this follows since if $K\subseteq L$ are two compact sets in $X$ then $h_d(\sigma, K)\leq h_d(\sigma, L)$ and any compact set in $X$ is contained in a compact open set).
\end{remark}

The following hypothesis is satisfied by a large number of Deaconu-Renault systems and is sufficient to guarantee that the sequence $(ssep(n,\varepsilon, \sigma, K))_n$ is monotone (for compact open sets $K$).

\begin{hypothesis}\label{niceone} Let $(X,\sigma)$ a Deaconu-Renault system. If $x\in X$ and $n\in \N$ are such that $I_n(x)=\{0,\ldots ,n-1\}$ then, for each open neighborhood $U$ of $x$, there exists $y\in U$ such that $I_{n+1}(y)=\{0,\ldots ,n\}$. 
\end{hypothesis}

\begin{remark} A Deaconu-Renault system $(X,\sigma)$ satisfies Hypotheses \ref{niceone} if, and only if, for each $n\in \N$, $x\in Dom(\sigma^{n-1})$, and open neighborhood $U$ of $x$, there exists $y\in U$ such that $y\in Dom(\sigma^n)$. This condition happens, for example, if each $Dom(\sigma^n)$ is dense in $X$. 
\end{remark}

\begin{proposition}
 Let $(X,\sigma)$ be a Deaconu-Renault system satisfying Hypothesis~\ref{niceone} and let $K$ be a compact open subset of $X$. Then, $ssep(n,\varepsilon, \sigma, K)\leq ssep(n+1,\varepsilon, \sigma, K)$.
\end{proposition}

\begin{proof}
Let $F\subseteq Dom(\sigma^{n-1})$ be a closed set and $S$ be an $(n,\varepsilon,\sigma, K\cap F)$-separated set of maximum cardinality. Notice that $I_n(s)=\{0,\ldots , n-1\}$, for each $s\in S$ (since $S\subseteq Dom(\sigma^{n-1})$). By the definition of an $(n,\varepsilon,\sigma, K\cap F)$-separated set, for each $x,y\in S$ there exists $0\leq j \leq n-1$, such that $d(\sigma^j(x), \sigma^j(y))>\varepsilon$. Furthermore, by Lemma~\ref{sirigordo}, $S$ is a finite set, and so we write $S=\{s_1,\ldots  s_k\}$. 

We now use Lemma~\ref{lema d_n}: for each $i\in \{1,\ldots ,k\}$, there exists open sets $V_i\subseteq K$ (since $K$ is compact-open) such that $s_i\in V_i$ and, moreover, for $z_i\in V_i$ and $z_j\in V_j$, we have that $d_n(z_i,z_j)>\varepsilon$ for each $i,j\in \{1,\ldots ,k\}$ with $i\neq j$. For each $i\in \{1,\ldots ,k\}$, choose $w_i\in V_i$ such that $I_{n+1}(w_i)=\{0,\ldots ,n\}$. Such $w_i$ exists since $(X,\sigma)$ satisfies Hypothesis \ref{niceone}. Now, notice that $\{w_1,..,w_k\}$ is an $(n+1,\varepsilon, \sigma, K\cap \{w_1,\ldots ,w_k\})$ separated set and, therefore, 
  $$sep(n,\varepsilon, \sigma, K\cap F) \leq sep(n+1,\varepsilon, \sigma, K\cap \{w_1,\ldots ,w_k\}).$$

Since $sep(n+1,\varepsilon, \sigma, K\cap \{w_1,\ldots ,w_k\})\leq ssep(n+1,\varepsilon, \sigma, K)$, we have that $sep(n,\varepsilon, \sigma, K\cap F) \leq ssep(n+1,\varepsilon, \sigma, K)$ for each closed set $F\subseteq Dom(\sigma^{n-1})$ and hence
$$ssep(n,\varepsilon, \sigma, K) \leq ssep(n+1,\varepsilon, \sigma, K).$$
\end{proof}

Invariance by conjugacy is one of entropy's most important features. In the remainder of the section, we describe this property for Deaconu-Renault systems. Since we are dealing with metric entropy, we first show how equivalence of metrics behaves in terms of entropy.

\begin{definition}
Let $d_1$ and $d_2$ be two metrics on $X$. 
We say that $d_1$ and $d_2$ are continuously equivalent if $Id_X:(X,d_1)\to (X,d_2)$ and $Id_X:(X,d_2)\to (X,d_1)$ are continuous. 
We say that $d_2$ is uniformly continuous with respect to $d_1$ if $Id_X:(X,d_1)\rightarrow (X,d_2)$ is uniformly continuous.
If both maps are uniformly continuous, then we say that the metrics are uniformly equivalent.
\end{definition}

\begin{remark}
If $d_1$ and $d_2$ are continuously equivalent, then a map $f:(X,d_1)\to (X,d_1)$ is continuous if, and only if, $f:(X,d_2)\to (X,d_2)$ is continuous. Therefore, if $((X,d_1),\sigma)$ is a Deaconu-Renault system and $d_1$ is continuously equivalent to $d_2$, then $((X,d_2),\sigma)$ is also a Deaconu-Renault system. Nevertheless, entropy does not need to be preserved by continuously equivalent metrics, see \cite[Remark~15, page 171]{Walters}.
\end{remark}

As it happens for the entropy of partial actions of $\Z$ (see \cite{Entropyforhomeo}), the entropy of Deaconu-Renault systems is preserved under uniformly equivalent metrics. We make this precise below. 

 \begin{proposition}
\label{invariance by the metric} Let $(X,\sigma)$ be a Deaconu-Renault system on the metric space $(X,d)$, and let $d'$ be a metric on $X$ such that $d$ and $d'$ are uniformly equivalent metrics on 
 $X$. Then $h_{d}(\sigma)=h_{d'}(\sigma)$.
 \end{proposition}
\begin{proof}

The proof is the same as in \cite[Lemma~3.15]{Entropyforhomeo}.
\end{proof}

Below we provide an example that shows that uniform equivalence of just one metric with respect to the other is not enough for invariance of the entropy.

\begin{example}\label{ex1} Consider the metric space $(\R,d_1)$ with the discrete metric $d_1$, and the metric space $(\R,d_2)$ with the usual metric $d_2$. The map $Id_X:(\R,d_{1})\to(\R,d_{2})$ is uniformly continuous, but $Id_X:(\R,d_{2})\to(\R,d_{1})$ is not. Let $\sigma:\R \to \R$ be the homeomorphism $\sigma(x)=2x$. Then, $h_{d_2}(\sigma)= \log(2)$ and $h_{d_1}(\sigma)=0$ (since each compact set $K$ in $(\R,d_1)$ is finite).
\end{example}

To finish this section, we prove that entropy is an invariant for conjugacies of Deaconu-Renault systems that are uniformly continuous, with inverses that are also uniformly continuous.

\begin{theorem}
Let $(X,\sigma_X)$ and $(Y,\sigma_Y)$ be Deaconu-Renault systems, where $(X,d)$ and $(Y,\ro)$ are metric spaces. Let $ \phi:X\rightarrow Y$ be a conjugacy between $(X,\sigma_X)$ and $(Y,\sigma_Y)$ such that $\phi$ and $\phi^{-1}$ are uniformly continuous. Then $h_d(\sigma_X)=h_{\ro}(\sigma_Y)$. 
\end{theorem}

\begin{proof} Let $d'$ be the metric on $Y$ defined by $d'(y_1,y_2)=d(\phi^{-1}(y_1), \phi^{-1}(y_2)$, and let $\phi_{d,d'}: (X,d)\rightarrow (Y,d')$ be given by $\phi_{d,d'}(x)=\phi(x)$. Clearly, $\phi_{d,d'}$ is an isometric isomorphism. Furthermore, since $\phi$ and $\phi^{-1}$ are uniformly continuous, we obtain that $Id=\phi_{d,d'}\circ \phi^{-1}$ and $Id=\phi\circ\phi_{d,d'}^{-1}$ are  uniformly continuous maps from $(Y,\ro)$ to $(Y,d')$ and from $(Y,d')$ to $(Y,\ro)$, respectively. 
It follows, from Proposition~\ref{invariance by the metric}, that $h_{d'}(\sigma_Y)=h_{\ro}(\sigma_Y)$. 

To finish, we have to prove that $h_d(\sigma_X)=h_{d'}(\sigma_Y)$. As we mentioned above, $\phi_{d,d'}$ is an isometric isomorphism from $(X,d)$ to $(Y,d')$. Notice that $I_n(x)=I_n(\phi_{d,d'}(x))$ for all $x\in X$ and all $n\in \N$. Indeed, if $i\in I_n(x)$ then $x\in Dom(\sigma_X^i)$. So, from Lemma~\ref{lemaconjuga}, $\phi_{d,d'}(x)\in Dom(\sigma_Y^i)$, so that, $i\in I_n(\phi_{d,d'}(x))$. Hence, for all $x,y\in X$ 
\begin{align*}
 d'_n(\phi_{d,d'}(x),\phi_{d,d'}(y))=&d'_n(\phi(x),\phi(y))\\
=&\displaystyle\max_i d'(\sigma_Y^i(\phi(x)),\sigma_Y^i(\phi(y)))  \\
=&\displaystyle\max_i d'(\phi(\sigma_X^i(x)),\phi(\sigma_X^i(y))) \\
=&\displaystyle\max_i d(\sigma_X^i(x),\sigma_X^i(y)) \\
=& d_n(x,y). 
\end{align*}

Now, for each $\varepsilon>0$, $n\in\N$, $K\subseteq X$ compact and $F\subseteq Dom(\sigma_X^{n-1})$, an $(n,\varepsilon,\sigma_X,K\cap F)$-separated set gives rise to an $(n,\varepsilon,\sigma_Y,\phi_{d,d'}(K) \cap\phi_{d,d'}(K\cap F))$-separated set with the same cardinality. Thus, $\mbox{sep}(n,\varepsilon,\sigma_X,K\cap  F)\leq \mbox{sep}(n,\varepsilon,\sigma_Y,\phi_{d,d'}(K) \cap\phi_{d,d'}(K\cap F))$. This implies that $\mbox{sep}(n,\varepsilon,\sigma_X,K)\leq\mbox{sep}(n,\varepsilon,\sigma_Y,\phi_{d,d'}(K))$ and hence $h_{\varepsilon}(\sigma_X,K,d)\leq h_{\varepsilon}(\sigma_Y,\phi_{d,d'}(K),d')$. Therefore, $h_{d}(\sigma_X,K)\leq h_{d'}(\sigma_Y,\phi_{d,d'}(K))$ and hence $h_d(\sigma_X)\leq h_{d'}(\sigma_Y)=h_{\ro}(\sigma_Y)$. Analogously, one shows that $h_d(\sigma_X)\geq h_{\ro}(\sigma_Y)$.

\end{proof}

\section{Factor maps} 

In the previous section, we have described invariance of the metric entropy in terms of uniformly continuous conjugacies. In this section, we propose a definition of a factor map between two Deaconu-Renault systems and show that the metric entropy always decreases under a uniformly continuous factor map. 

\begin{definition}
Let $X, Y$ be topological spaces. A map $f:X\rightarrow Y$ is called a compact covering if, for every compact subset $K\subseteq Y$, there exists some compact subset $C\subseteq X$ such that $f(C)=K$. 
\end{definition}

For example, a proper map (that is, a map $f:X\rightarrow Y$ such that $f^{-1}(K)$ is compact for each compact set $K\subseteq Y$) is a compact covering.

\begin{definition}
Let $(X,\sigma_X)$ and $(Y,\sigma_Y)$ be Deaconu-Renault systems. We say that  $\phi:X\rightarrow Y$ is a factor map if it is a continuous, compact covering, surjective map such that
 $\phi(Dom(\sigma_X))\subseteq Dom(\sigma_Y)$,
$\sigma_Y(\phi(x)) = \phi(\sigma_X(x))$ for all $x\in Dom(\sigma_X) $, and  $\phi^{-1}(Dom(\sigma_Y))\subseteq Dom(\sigma_X)$.
\end{definition}

\begin{remark} Since $\phi^{-1}(Dom(\sigma_Y))\subseteq Dom(\sigma_X)$, we have that $Dom(\sigma_Y)\subseteq \phi(Dom(\sigma_X))$. Moreover, since $\phi(Dom(\sigma_X))\subseteq Dom(\sigma_Y)$ we have that $\phi(Dom(\sigma_X))= Dom(\sigma_Y)$

\end{remark}

Our definition of factor map is compatible with the definition of a conjugacy given in Definition~\ref{def of conjugacy}. Indeed, if a factor map has a continuous inverse, then it is a conjugacy, as we show below. 

\begin{proposition}
Let $\phi$ be a factor map between $(X,\sigma_X)$ and $(Y,\sigma_Y)$. If $\phi$ is injective with continuous inverse, then it is a conjugacy.
\end{proposition}
\begin{proof}

All we need to show is that $\sigma_X \circ \phi^{-1}=\phi^{-1}\circ \sigma_Y$ in $Dom(\sigma_Y)$. Let $y \in Dom(\sigma_Y)$. Since $\phi$ is a factor map, we have that $Dom(\sigma_Y)= \phi (Dom(\sigma_X))$. So, there exists $x\in Dom(\sigma_X)$ such that $y=\phi(x)$. Since $\phi$ is a factor map, it commutes with the shift, that is, $\phi(\sigma_X(x))= \sigma_Y (\phi(x))$. Composing the two terms of the last equality on the left with $\phi^{-1}$, and using that $x=\phi^{-1}(y)$, we obtain that $\sigma_X(\phi^{-1}(y))= \phi^{-1}(\sigma_Y (y))$, as desired. 
\end{proof}

Next, we describe how a factor map preserves the domains of the iterates of the local homeomorphism $\sigma$.

\begin{lemma}\label{lemaconjuga2}
Let $(X,\sigma_X)$ and $(Y,\sigma_Y)$ be Deaconu-Renault systems and let $\phi\colon X \to Y$ be a factor map. Then, $\phi^{-1}(Dom(\sigma_{Y}^i)) \subseteq  Dom(\sigma_{X}^i)$
and $\sigma_Y^i(\phi(x))=\phi(\sigma_X^i(x))$, for all $i\in \N$ and $x\in Dom(\sigma_X^i)$. 
\end{lemma}
\begin{proof}
We prove the lemma by induction. For $k=1$, since $\phi$ is a factor, we have that $\phi^{-1}(Dom(\sigma_Y))\subseteq Dom(\sigma_X)$ and $\sigma_Y(\phi(x))=\phi(\sigma_X(x))$ for $x\in Dom(\sigma_X)$.

Suppose the hypothesis is true for $k$. We show it for $k+1$. First, we prove that  
$$\phi^{-1}(Dom(\sigma_Y^{k+1}))\subseteq Dom(\sigma_X^{k+1}).$$ Let $x\in \phi^{-1}(Dom(\sigma_Y^{k+1}))$. Then, $\phi(x)\in Dom(\sigma_Y^{k+1})\subseteq Dom(\sigma_Y^{k})$, and so $x\in \phi^{-1}(Dom(\sigma_Y^k))\subseteq Dom(\sigma_X^k)$ (where the last containment follows from the induction hypothesis). Again, from the induction hypothesis, we obtain that $\phi(\sigma_X^k(x))=\sigma_Y^k(\phi(x))\in Dom(\sigma_Y)$. Hence, 
 $\sigma_X^k(x)\in \phi^{-1}(Dom(\sigma_Y))\subseteq Dom(\sigma_X)$, which implies that $x\in Dom(\sigma_X^{k+1})$.

To finish, notice that $\phi(\sigma_X ^{k+1}(x))=\phi(\sigma_X^k(\sigma_X(x)))=\sigma_Y^k(\phi(\sigma_X(x)))=\sigma_Y^{k+1}(\phi(x))$.
\end{proof}

We can now prove that the metric entropy always decreases under uniformly continuous factor maps.

\begin{theorem}\label{factorentropy} Let $(X,\sigma_X)$ and $(Y,\sigma_Y)$ be Deaconu--Renault systems, and let $\phi\colon X \to Y$ be a uniformly continuous factor map. Then,
 $h_d(\sigma_X)\geq h_d(\sigma_Y) $.
\end{theorem}

\begin{proof}
Let $K_Y$ be a compact subset of $Y$. Since $\phi$ is a factor map, then there exists a compact set $K_X\subseteq X$ such that $\phi(K_X)=K_Y$.

Let $\varepsilon>0$, and choose let $0<\delta<\varepsilon$ such that if $d(x,y)<\delta$ then $d(\phi(x),\phi(y))<\epsilon$. Given $n\in \N$, we prove that $sspan (n,\varepsilon, \sigma_Y, K_Y)\leq sspan (n,\delta, \sigma_X, K_X)$.

Let $F$ be a closed subset of $Dom(\sigma^{n-1}_Y)$. Then, by Lemma~\ref{lemaconjuga2}, $\phi^{-1}(F)$ is a closed subset of $Dom(\sigma^{n-1}_X)$.

Chose $A$ as an $(n,\delta,\sigma_X, K_X\cap \phi^{-1}(F))$ spanning set. Then, $\phi(A)$ is an $(n,\varepsilon,\sigma_Y, K_Y\cap F)$ spanning set. Indeed, for each $z\in  K_Y\cap F$, let $x\in K_X\cap \phi^{-1}(F)$ be such that $z=\phi(x)$. 

Let $a\in A$ be such that $d_n(a,x)\leq \delta$ (so $\displaystyle \max_{i\in I_n(a)\cap I_n(x)}\{ d(\sigma^i(x),\sigma^i(a))\}\leq \delta$).
Then, by Lemma~\ref{lemaconjuga2}, we have that
\[ d_n(z,\phi(a))=\max_{i\in I_n(\phi(x))\cap I_n(\phi(a))}\{ d(\sigma^i(\phi(x)),\sigma^i(\phi(a)))\}=\max_{i\in I_n(x)\cap I_n(a)}\{d(\phi(\sigma^i(x)),\phi(\sigma^i(a)))\}\leq \varepsilon,
\]
where the last inequality follows from the uniform continuity of $\phi$. Hence, $\phi(A)$ is an $(n,\varepsilon,\sigma_Y, K_Y\cap F)$ spanning set, as claimed.

Now, since the cardinality of $\phi(A)$ is less or equal to the cardinality of $A$, it follows that 
$span(n,\varepsilon,\sigma_Y, K_Y\cap F)\leq span (n,\delta,\sigma_X, K_X\cap \phi^{-1}(F))\leq  sspan (n,\delta,\sigma_X, K_X) $. Taking supremum over closed sets $F$ contained in the domain of $\sigma^{n-1}_Y$, it follows that $sspan(n,\varepsilon,\sigma_Y, K_Y)\leq  sspan (n,\delta,\sigma_X, K_X) $. Since $\delta\rightarrow 0$ as $\varepsilon \rightarrow 0$, the above inequality implies that $h_d(\sigma_Y,K_Y)\leq h_d(\sigma_X,K_X) \leq h_d(\sigma_X)$. Taking supremum over all the compact sets $K_Y \subseteq Y$ we obtain that $h_d(\sigma_Y) \leq h_d(\sigma_X)$, as desired.
\end{proof}

\begin{example} Let $(\R, d_1)$ and $(\R,d_2)$ be the Deaconu-Renault systems of Example \ref{ex1}, where $d_1$ and $d_2$ are the discrete and usual metrics, respectively, and $\sigma:\R\rightarrow \R$ is the homeomorphism $\sigma(x)=2x$. Notice that $\phi:(\R,d_1)\rightarrow (\R,d_2)$, given by $\phi(x)=x$, is a uniformly continuous map which is not a factor map since it is not a compact covering map. In this case, $h_{d_2}(\sigma)=log(2)$ and $h_{d_1}(\sigma)=0$. So, the hypothesis that a factor map is a compact covering can not be dropped, if we want the entropy to decrease under uniformly continuous factors. 
\end{example}

\begin{remark}
In the literature of countable Markov shift spaces (in particular of shifts associated with locally finite graphs), there are two proposed entropy theories, one due to Salama and the other to Gurevic \cite{gurevich, Sal1, Wag2, LindMarcus}. In both cases there are examples where entropy can increase via factors (in this setting a factor is a continuous, surjective, shift commuting map), see \cite{LindMarcus, Sal1, Pet2} and \cite[Example 7.2.8]{Kitchens}. Later, in Section~\ref{graphEntropy}, we compute the entropy of Deaconu-Renault systems associated with graphs and show that, in the case of locally finite graphs, our proposed entropy coincides with the Gurevic entropy. So, for Theorem~\ref{factorentropy} the uniform continuity requirement is essential.
\end{remark}

\section{The entropy of invariant subsets}

In this short section, we define invariant subsets of a Deaconu-Renault system and its metric entropy. This is useful later, to compute the entropy of Deaconu-Renault systems associated with graphs that are not connected, for example.

\begin{definition}
Let $(X,\sigma)$ be a Deaconu-Renault system. We say that $Y \subseteq X$ is invariant if $\sigma(Dom(\sigma)\cap Y)\subseteq Y$. Given an open invariant subset $Y$, we call the system $(Y,\sigma|_Y)$ a restriction Deaconu-Renault system.
\end{definition} 

\begin{remark}
Notice that a restriction Deaconu-Renault system is always a Deaconu-Renault system itself, and so its metric entropy $h_d(\sigma|_Y)$ is well-defined.
\end{remark}

In the result below, we relate the metric entropy of a Deaconu-Renault system with the metric entropy of restriction systems. 

\begin{proposition}\label{prop:entropy-restriction}
Let $(X,\sigma)$ be a Deaconu-Renault system, $Y_i$ be open invariant subsets of $X$, and $(Y_i,\sigma|_{Y_i})$ be restrictions to $Y_i$, for $i=1,\ldots ,k$, such that the union $\bigcup\limits_{i=1}^k Y_i$ (not necessarily disjoint) is $X$. Then: 
\begin{enumerate}
    \item
$ h_d(\sigma) \geq\max\limits_{1\leq i \leq k} h_d( \sigma|_{Y_i}).
$
\item If each $Y_i$ is closed then $
h_d(\sigma) =\max\limits_{1\leq i\leq k} h_d( \sigma|_{Y_i}).
$
\item If $Y_i\cap Y_j=\emptyset$ for all $i\neq j$ then $
h_d(\sigma) =\max\limits_{1\leq i\leq k} h_d( \sigma|_{Y_i})
$.
\end{enumerate}

\end{proposition}
\begin{proof}
To simplify the computations, we assume $k=2$ in this proof. The general case then follows straightforwardly.

Let us prove the first item. By hypothesis, $X=Y_1\cup Y_2$. Let $K_i$ be a compact subset of $Y_i$ and $F_i$ be a closed subset of $Dom(\sigma|_{Y_i}^{n-1})$, $i=1,2$.
Since $X$ is a metric space and $Dom(\sigma|_{Y_i}^{n-1})$ is open, by Remark~\ref{uniao de fechados}, we can write $Dom(\sigma|_{Y_i}^{n-1}) = \bigcup\limits_{j=1}^{\infty} G_j^i$, where $G_j^i$ are closed sets of $X$ and $G_j^i\subseteq G_{j+1}^i$ for all $j$. Because each $Y_i$ has the subspace topology, we have that $F_i = F_i' \cap Y_i$, where $F_i'$ is a closed set of $X$. Furthermore, $F_i'\cap G_j^i$ is a closed set of $X$ and is contained in $Dom(\sigma|_{Y_i}^{n-1})$ for all $j$ and each $i=1,2$. 

Now, let $A$ be an $(n,\varepsilon,\sigma|_{Y_i},K_i\cap F_i)$-separated set of maximum cardinality in $Y_i$. Then, $A$ is a finite set, and there exists $j_0$ such that every element in $A$ belongs to $F_i'\cap G_{j_0}^i$.
 So, $A$ is an $(n,\varepsilon,\sigma,K_i\cap (F_i'\cap G_{j_0}^i))$-separated set and hence
  $$sep(n,\varepsilon,\sigma|_{Y_i},K_i\cap F_i)\leq sep (n,\varepsilon,\sigma,K_i\cap (F_i'\cap G_{j_0}^i)) \leq ssep (n,\varepsilon,\sigma,K_i). $$
Therefore, $\mbox{ssep}(n,\varepsilon,\sigma|_{Y_i},K_i)\leq\mbox{ssep}(n,\varepsilon,\sigma,K_i)$,
and $h_d( \sigma ) \geq\max h_d( \sigma|_{Y_i})$, as desired.

Now we prove the second item. 

Let $K\subset X$ be a compact set, $F\subseteq Dom(\sigma^{n-1})$ a closed set and $A\subset K\cap F$ an $(n,\varepsilon,\sigma,K\cap F)$-separated set of maximum cardinality. Let $K_i=K\cap Y_i$, which is compact since each $Y_i$ is closed. Since $A$ is finite, we have that $A=(A\cap Dom(\sigma|_{Y_1}^{n-1}))\cup (A\cap Dom(\sigma|_{Y_2}^{n-1}))$, where the sets $G_j^i$ are as in the proof of the first item. 

Then,  
\begin{align*}(n,\varepsilon,\sigma,K\cap F)&\leq \mbox{sep}(n,\varepsilon,\sigma|_{Y_1},K_1\cap (F\cap G_{k_1}^1))+\mbox{sep}(n,\varepsilon,\sigma|_{Y_2},K_2\cap (F\cap G_{k_2}^2)) \\ &
\leq 2\max_{1\leq i\leq 2}\mbox{sep}(n,\varepsilon,\sigma|_{Y_i},K_i\cap (F\cap G_{k_i}^i))
\leq 2\max_{1\leq i\leq2}\mbox{ssep}(n,\varepsilon,\sigma|_{Y_i},K_i).
\end{align*}

Taking the supremum over all closed subsets of $Dom(\sigma^{n-1})$, we obtain that 
$$
\mbox{ssep}(n,\varepsilon,\sigma,K)\leq 2\max_{1\leq i\leq2}\mbox{ssep}(n,\varepsilon,\sigma|_{Y_i},K_i). 
$$ 
Therefore,
\begin{align*}
h_{\varepsilon}(\sigma,K,d)&=\limsup_{n\to\infty}\frac{1}{n} \log \mbox{ssep}(n,\varepsilon,\sigma,K)\\
                           &\leq
\limsup_{n\to\infty}\frac{1}{n}\log 2+ \limsup_{n\to\infty}\frac{1}{n} \log\max_{1\leq i\leq 2}\mbox{ssep}(n,\varepsilon,\sigma|_{Y_i},K_i)
\\&=\max_{1\leq i\leq 2}h_{\varepsilon}(\sigma|_{Y_i},K_i,d).
\end{align*}
Letting $\varepsilon\to 0$ and taking the supremum over all compact subsets of $X$, we conclude that
$h_d(\sigma)\leq \max\limits_{1\leq i\leq 2}h_d(\sigma|_{Y_i})$. By the first item, we know that $h_d(\sigma)\geq \max\limits_{1\leq i\leq 2} h_d(\sigma|_{Y_i})$ and hence  $h_d(\sigma)= \max\limits_{1\leq i\leq 2}h_d(\sigma|_{Y_i})$, as desired.

The third item of the proposition follows from the second one, since if the subsets $Y_j$ are disjoint then they are closed. 
\end{proof}

\section{Entropy via open covers}\label{topoent}

In this section, we generalize the classical notion of entropy via covers to include Deaconu-Renault systems. For metrizable spaces, we show that the entropy of this section is a lower bound to the metric entropy. We start recalling the definition of the join of two covers.

\begin{definition}
Let $X$ be a topological space, $Y,Z\subseteq X$, and let $\alpha,\beta$ be covers of $Y$ and $Z$, respectively. The join of $\alpha$ and $\beta$, denoted by $\alpha\vee\beta$, is the set of all the intersections $A\cap B$, where $A\in \alpha$ and $B\in \beta$. 
\end{definition}

\begin{remark} In the previous definition, $\alpha\vee\beta$ is a cover of $Y\cap Z$. Moreover, if $\alpha,\beta$ are open covers of $Y$ and $Z$, then $\alpha\vee \beta$ is an open cover of $Y\cap Z$.
\end{remark}

\begin{definition} Let $X$ be a topological space, $Y\subseteq X$ and $\alpha$ be a cover of $Y$. We denote by $N(\alpha,Y)$ the smallest cardinality among the cardinalities of sub-covers of $\alpha$. If $N(\alpha,Y)<\infty$ then we define $H(\alpha,Y)=log(N(\alpha,Y))$, and if $N(\alpha,Y)=\infty$ then we define $H(\alpha,Y)=\infty$.
\end{definition}

Let $(X,\sigma)$ be a Deaconu-Renault system, $Y\subseteq X$ and let $\alpha$ be a cover of $Y$. For each $i\in\N$, the family of all the sets $\sigma^{-i}(A)$, with $A\in \alpha$, is a cover of $\sigma^{-i}(Y)$. We denote this cover by $\sigma^{-i}(\alpha)$. Moreover, for each $n\in \N$ we define 
$$\alpha_n:=\alpha\vee\sigma^{-1}(\alpha)\vee\ldots \vee\sigma^{-n}(\alpha),$$
which is a cover of 
$$Y_n:=Y\cap \sigma^{-1}(Y)\cap\ldots \cap \sigma^{-n}(Y).$$

 \begin{remark}
 It is possible that $Y_n=\emptyset$, for some $n\in \N$. In this case, $N(\alpha_n,Y_n)=1$ and $H(\alpha_n,Y_n)=log(N(\alpha_n,Y_n))=log(1)=0$. Furthermore, in this situation, we also have that $Y_m=\emptyset,$ for each $m\geq n$, and hence $H(\alpha_m,Y_m)=0$ for each $m\geq n$. It is also possible that $N(\alpha_n,Y_n)=\infty$, in which case we have that $H(\alpha_n,Y_n)=\infty$.
\end{remark}

The following lemma describes properties of $N(\alpha_n, Y_n)$.

\begin{lemma}\label{lemma02} Let $(X,\sigma)$ be a Deaconu-Renault system, $Y,Z\subseteq X$, $\alpha$ a cover of $Y$, $\beta$ a cover of $Z$, and $m,n\in \N$. Then, the following inequalities are true.
\begin{enumerate}
    \item $N(\alpha \vee \beta,Y\cap Z)\leq N(\alpha,Y)N(\beta,Z)$.
    \item $N(\sigma^{-n}(\alpha)_m,\sigma^{-n}(Y)_m)\leq N(\alpha_m,Y_m).$
\item $N(\alpha_{m+n},Y_{n+m})\leq N(\alpha_n,Y_n)N(\alpha_m,Y_m).$
\item $N(\beta_n,Y_n)\leq N(\beta_n,Z_n)$, if $Y\subseteq Z$.

\end{enumerate}
\end{lemma}

\begin{proof} We begin with the proof of the first item. Notice that the inequality $$N(\alpha \vee \beta, Y\cap Z)\leq N(\alpha,Y)N(\beta,Z)$$ is obviously true if $N(\alpha,Y)=\infty$ or $N(\beta,Z)=\infty$. So, we may suppose that $N(\alpha,Y)<\infty$ and $N(\beta,Z)<\infty$. Let $N(\alpha,Y)=p$, $N(\beta,Z)=q$, and let $\{A_1,\ldots ,A_p\}$ and $\{B_1,\ldots ,B_q\}$ be subcovers of $\alpha$ and $\beta$, respectively. Then, $\{A_i\cap B_j:1\leq i\leq p, 1\leq j\leq q\}$, whose cardinality is less or equal to $N(\alpha,Y)N(\beta,Z)$,  is a subcover of $\alpha\vee \beta$. Therefore, $N(\alpha\vee\beta,Y\cap Z)\leq N(\alpha,Y)N(\beta,Z)$.

Next, we prove the second item.
Notice that for each $B\in \alpha_m$, $B$ has the form $$B=B_0\cap B_1\cap\ldots \cap B_m,$$ where $B_i\in \sigma^{-i}(\alpha)$ for each $i\in \{0,\ldots ,m\}$. Then, $\sigma^{-n}(B_i)\in \sigma^{-n}(\sigma^{-i}(\alpha))=\sigma^{-i}(\sigma^{-n}(\alpha))$ for each $i\in \{0,\ldots ,m\}$, and so $$\sigma^{-n}(B)=\sigma^{-n}(B_0\cap B_1\cap \ldots  \cap B_m)=\sigma^{-n}(B_0)\cap \sigma^{-n}(B_1)\cap\ldots \cap \sigma^{-n}(B_m),$$
which belongs to  $\sigma^{-n}(\alpha)\vee  \sigma^{-1}(\sigma^{-n}(\alpha))\vee\ldots \vee \sigma^{-m}(\sigma^{-n}(\alpha))=\sigma^{-n}(\alpha)_m.$
Therefore, for each $B\in \alpha_m$, we have that $\sigma^{-n}(B)\in \sigma^{-n}(\alpha)_m$.

Now, we prove the desired inequality.
If $N(\alpha_m,Y_m)=\infty$ we are done. So, suppose that $N(\alpha_m,Y_m)=p<\infty$. 
Let $\{A_1,\ldots ,A_p\}$ be a subcover of $\alpha_m$. Then, $$Y\cap \sigma^{-1}(Y)\cap \ldots \cap \sigma^{-m}(Y) \subseteq \bigcup\limits_{i=1}^pA_i,$$ from where we conclude that
$$\sigma^{-n}(Y)\cap \sigma^{-n-1}(Y)\cap \ldots \cap \sigma^{-n-m}(Y) \subseteq \bigcup\limits_{i=1}^p\sigma^{-n}(A_i).$$
Since $A_i \in \alpha_m$, from the paragraph above, we obtain that $\sigma^{-n}(A_i)\in \sigma^{-n}(\alpha)_m$. Therefore, $ \{\sigma^{-n}(A_1),\ldots ,\sigma^{-n}(A_p)\}$ 
is a subcover of $\sigma^{-n}(\alpha)_m$ and hence  $N(\sigma^{-n}(\alpha)_m,\sigma^{-n}(Y)_m)\leq N(\alpha_m,Y_m)$ as desired.

Now we prove the third item. Write 

$$\alpha_{n+m}=[\alpha\vee \sigma^{-1}(\alpha)\vee\ldots \vee \sigma^{-n}(\alpha)]\vee [\sigma^{-n-1}(\alpha)\vee \ldots  \vee \sigma^{-n-m}(\alpha)]=$$
$$=\alpha_n \vee \sigma^{-n-1}(\alpha)_m.$$ 
From the first item, we obtain that $N(\alpha_{n+m},Y_{n+m})\leq N(\alpha_n,Y_n)N(\sigma^{-n-1}(\alpha)_m,\sigma^{-n-1}(Y)_m)$ and, from the second item, we have that  $N(\sigma^{-n-1}(\alpha)_m,\sigma^{-n-1}(Y)_m)\leq N(\alpha_m,Y_m)$. Hence, $N(\alpha_{n+m},Y_{n+m})\leq N(\alpha_n,Y_n)N(\alpha_m,Y_m)$.

To prove the last item, let $\beta$ be a cover of $Z$ and, for a fixed $n\in \N$, consider the cover $\beta_n$ of $Z_n$. Let $\gamma$ be a subcover of $\beta_n$. Since $Y\subseteq Z$, we have that $Y_n\subseteq Z_n$, and hence $\gamma$ also covers $Y_n$. Therefore, $N(\beta_n,Y_n)\leq N(\beta_n,Z_n)$.
\end{proof}

To define the entropy via covers of a Deaconu-Renault system $(X,\sigma)$,  we will deal with compact sets $K$ in $X$. But even when $K$ is compact, $K_n$ may not be compact, for some $n\in \N$. For example, consider $(\R,\sigma)$, where $\R$ has the usual topology and $\sigma:(0,\infty)\rightarrow \R$ is given by $\sigma(x)=x$. For the compact set $K=[0,1]$, we get $K_n=(0,1]$ for each $n\geq 1$, which is not compact. Despite this, according to the following lemma, for each compact set $K$ in a Deaconu-Renault system and each open cover $\alpha$ of $K$, the cover $\alpha_n$ of $K_n$ has a finite subcover.

\begin{lemma}\label{lemmacover}
Let $(X,\sigma)$ be a Deaconu-Renault system, let $K\subseteq X$ be compact, and let $\alpha$ be an open cover of $K$. Then, for each $n\in \N$, $\alpha_n$ is an open cover of $K_n$, and $\alpha_n$ has a finite subcover.
\end{lemma}

\begin{proof}
 Let $K\subseteq X$ be a compact set, $\alpha$ be an open cover of $K$, and let $i\in \N$. For $A\in \alpha$, since $A$ is open in $X$, we have that $\sigma^{-i}(A)$ is an open set in $Dom(\sigma^i)$ and, since $Dom(\sigma^i)$ is open in $X$, we have that $\sigma^{-i}(A)$ is open in $X$. Therefore, $\sigma^{-i}(\alpha)$ is an open cover (in $X$) of $\sigma^{-i}(K)$, and so $\alpha_n$ is an open cover (in $X$) of $K_n$ for each $n\in \N$. Since $K$ is compact, there exists a finite subcover $\{A_1,\ldots ,A_m\}$ of $\alpha$. Then, $\{\sigma^{-i}(A_1),\ldots ,\sigma^{-i}(A_m)\}$ is a finite subcover of $\sigma^{-i}(\alpha)$, covering $\sigma^{-i}(K)$ and hence $\alpha_n$ has a finite subcover. 
\end{proof}

We need one last result before we can define entropy via covers of a Deaconu-Renault system.

\begin{proposition}\label{limitexists}
Let $(X,\sigma)$ be a Deaconu-Renault system, $K\subseteq X$ compact and $\alpha$ an open cover of $K$. Then, $$\lim\limits_{n\rightarrow \infty}\frac{1}{n}H(\alpha_n,K_n)$$ exists.
\end{proposition}

\begin{proof} By Lemma \ref{lemmacover}, 
we have that $N(\alpha_n,K_n)<\infty$ for each $n\in \N$. 
From the third item in Lemma~\ref{lemma02}, we obtain that $H(\alpha_{n+m},K_{n+m})\leq H(\alpha_n,K_n)+H(\alpha_m,K_m)$ for each $n,m\in \N$. With this, we can now apply  \cite[Theorem 4.9]{Walters} and conclude that $\lim\limits_{n\rightarrow \infty}\frac{1}{n}H(\alpha_n,K_n)$ exists.
\end{proof}

\begin{definition}\label{defofentropyviacovers} Let $(X,\sigma)$ be a Deaconu-Renault system. 
\begin{itemize}
    \item For a compact set $K\subseteq X$ and an open cover $\alpha$ of $K$, the \textbf{entropy of $\sigma$ relative to the open cover $\alpha$ of $K$} is defined by:
    $$h(\alpha,\sigma,K)=\lim\limits_{n\rightarrow \infty}\frac{1}{n}H(\alpha_n,K_n).$$

\item For a compact set $K\subseteq X$, the \textbf{entropy of $\sigma$ relative to $K$} is defined by:
    $$h(\sigma,K)=\sup\limits_{\alpha}h(\alpha,\sigma,K),$$
    where the supremum is taken over all the open covers $\alpha$ of $K$.
    
\item The \textbf{(topological) entropy of $\sigma$} is defined by
$$h(\sigma)=\sup\limits_{K}h(\sigma, K),$$
where the supremum is taken over all the compact sets $K\subseteq X$. 
\end{itemize}
\end{definition}

\begin{remark} Notice that if $K$ is a compact set in $X\setminus Dom(\sigma)$, then $K\cap \sigma^{-1}(K)=\emptyset$ and therefore $K_n=\emptyset$ for each $n\geq 1$. Then, for each open cover $\alpha$ of $K$ and for each $n\geq 1$, we have that  $N(\alpha_n,K_n)=1$ and hence $H(\alpha_n,K_n)=0$. Consequently,  $h(\sigma, K)=0$. So, in Definition~\ref{defofentropyviacovers}, the supremum can be taken over all the compact sets $K$ with a nonempty intersection with $Dom(\sigma)$. 
\end{remark}

For Deaconu-Renault systems such that $Dom(\sigma)$ is a closed set in $X$, we prove below that the last supremum in Definition~\ref{defofentropyviacovers} can be taken over all the compact sets contained in $Dom(\sigma)$, instead of over all compact sets. 

\begin{proposition}\label{equiv01} Let $(X,\sigma)$ be a Deaconu-Renault system with $Dom(\sigma)$ closed in $X$. Then, in Definition \ref{defofentropyviacovers}, we can take the supremum over all the compact sets $L\subseteq Dom(\sigma)$, instead of over all the compact sets $K\subseteq X$.
\end{proposition}

\begin{proof} 
All we need to prove is that 
$$\sup\{h(\sigma, K):K\text{ is a compact set in } X\}\leq \sup\{h(\sigma, L):L\text{ is a compact set in } Dom(\sigma)\},$$ since the other inequality is straightforward. 

Let $K\subseteq X$ be a compact set. Since $K$ is compact, $\sigma^{-1}(K)$ is closed in $Dom(\sigma)$ and hence closed in $X$ (since $Dom(\sigma)$ is closed in $X$). Therefore, $L=K\cap \sigma^{-1}(K)$ is compact.

Notice that for each $n\in \N$, 
\begin{eqnarray*}L_n &=& L\cap \sigma^{-1}(L)\cap\ldots \cap \sigma^{-n}(L)\\
&=&(K\cap \sigma^{-1}(K))\cap \sigma^{-1}(K\cap \sigma^{-1}(K))\cap\ldots \cap \sigma^{-n}(K\cap \sigma^{-1}(K)) \\
&=&K\cap \sigma^{-1}(K)\cap\ldots \cap \sigma^{-n-1}(K)=K_{n+1}.
\end{eqnarray*}

Let $\alpha$ be an open cover of $K$, and define $\beta=\{A\cap \sigma^{-1}(B):A,B\in \alpha\}$, which is an open cover of $L$. For $n\in \N$, each element of $\beta_n$ is of the form 
$$\left(A_0\cap \sigma^{-1}(B_1)\right)\cap \sigma^{-1}\left(A_1\cap \sigma^{-1}(B_2)\right)\cap\sigma^{-2}\left(A_2\cap \sigma^{-1}(B_3)\right)\ldots \cap \sigma^{-n}\left(A_{n}\cap \sigma^{-1}(B_n)\right)=$$
$$=A_0\cap \sigma^{-1}\left(B_1\cap A_1\right)\cap \sigma^{-2}\left(B_2\cap A_2\right)\cap \ldots \cap \sigma^{-n}\left(B_n\cap A_n\right)\cap \sigma^{-n-1}(B_n),$$ where $A_i,B_j\in \alpha$, and hence is a subset of 
$$=A_0\cap \sigma^{-1}\left(A_1\right)\cap \sigma^{-1}\left(A_2\right)\cap \ldots \cap \sigma^{-n}\left(A_n\right) \cap\sigma^{-n-1}(B_n),$$ which is an element of $\alpha_{n+1}$. So, we have shown that each element of $\beta_n$ is a subset of an element of $\alpha_{n+1}$.

Let $\gamma$ be a subcover of $\beta_n$ of minimal cardinality, covering $L_n$, and choose, for each $U\in \gamma$, an element $V_U\in \alpha_{n+1}$ such that $U\subseteq V_U$. Since $K_{n+1}=L_n$, we have that $\{V_U:U\in \gamma\}$ is a subcover of $\alpha_{n+1}$, which covers $K_{n+1}$. Therefore, $N(\alpha_{n+1},K_{n+1})\leq N(\beta_n,L_n)$ for each $n\in \N$ and so
\begin{eqnarray*} 
h(\alpha,\sigma, K)&=&\lim\limits_{n\rightarrow \infty}\frac{1}{n+1}H(\alpha_{n+1}, K_{n+1})=\lim\limits_{n\rightarrow \infty}\frac{n}{n+1}\frac{1}{n}H(\alpha_{n+1}, K_{n+1})\\
&\leq& \lim\limits_{n\rightarrow \infty}\frac{1}{n}H(\beta_n, L_n)=h(\beta,\sigma, L).
\end{eqnarray*}

We have proved above that for each open cover $\alpha$ of $K$ there exists an open cover $\beta$ of $L$ such that $h(\alpha,\sigma, K)\leq h(\beta,\sigma, L)$. Therefore, 
$$\sup\{h(\alpha, \sigma,K): \alpha \text{ is an open cover of }K\}\leq\sup\{h(\beta, \sigma,L): \beta \text{ is an open cover of }L\},$$ which means that $h(\sigma, K)\leq h(\sigma, L)$. 

Wrapping it up, we have proved that for each compact set $K\subseteq X$, there exists a compact set $L\subseteq Dom(\sigma)$ such that $h(\sigma, K)\leq h(\sigma, L)$ and therefore,
$$\sup\{h(\sigma, K):K\text{ is a compact set in }X\}\leq \sup\{h(\sigma, L):L\text{ is a compact set in }Dom(\sigma)\},$$ as desired.

\end{proof}

Next, we show that the entropy of $\sigma$ relative to compact sets is an increasing function.

\begin{lemma}\label{entropia crescente} Let $(X,\sigma)$ be a Deaconu-Renault system and let $K$ and $C$ be compact subsets of $X$ such that $K\subseteq C$. Then, $h(\sigma, K)\leq h(\sigma,C)$.
\end{lemma}

\begin{proof} To prove the desired inequality, it is enough to show that for each open cover $\alpha$ of $K$ there exists an open cover $\beta$ of $C$ such that $h(\alpha, \sigma, K)\leq h(\beta,\sigma, C)$, which in turn follows if the covers are such that $N(\alpha_n,K_n)\leq N(\beta_n,C_n)$ for each $n\in \N$. 

Let $\alpha$ be an open cover of $K$. Define $\beta=\alpha\cup\{K^c\}$ (where $K^c=X\setminus K$), and notice that $\beta$ is an open cover of $C$. We show that, for each $n\in \N$, $N(\alpha_n,K_n)\leq N(\beta_n, C_n)$. Fix $n\in \N$. Notice that $\beta$ is also an open cover of $K$ and let $\gamma$ be a subcover of $\beta_n$, covering $K_n$, with cardinality equal to $N(\beta_n,K_n)$ (which is finite by Lemma \ref{lemmacover}). Let $A=A_0\cap A_1\cap ..\cap A_n$, where $A_i\in \sigma^{-i}(\beta)$ for each $i\in \{0,1,\ldots ,n\}$, be an element of $\gamma$.
Suppose that $A_j=\sigma^{-j}(K^c)$ for some $j\in \{0,1,\ldots ,n\}$. Then, $A_j\cap \sigma^{-j}(K)=\sigma^{-j}(K^c)\cap \sigma^{-j}(K)=\emptyset$, and hence 
$$A\cap K_n=\left(A_0\cap A_1\cap \ldots \cap A_n\right)\bigcap\left( K\cap \sigma^{-1}(K)\cap \ldots \cap \sigma^{-n}(K)\right)=\emptyset,$$ which is impossible since $\gamma$ covers $K_n$ and has minimal cardinality. We conclude that $A_i\in \sigma^{-i}(\alpha)$ for each $i\in \{0,1,\ldots ,n\}$, and hence $A\in \alpha_n$. Therefore, $\gamma\subseteq \alpha_n$ and so  $N(\alpha_n,K_n)\leq N(\beta_n,K_n)$. By the last item of Lemma~\ref{lemma02}, we have that $N(\beta_n,K_n)\leq N(\beta_n,C_n)$ and hence $N(\alpha_n,K_n)\leq N(\beta_n,C_n)$ as desired.
\end{proof}

\begin{corollary}
Let $(X,\sigma)$ be a Deaconu-Renault system. If $Dom(\sigma)$ is compact, then $$h(\sigma)=\sup\limits_\alpha\{ h(\alpha,\sigma, Dom(\sigma))\},$$ where the supremum is taken over all the open covers of $Dom(\sigma)$.
\end{corollary}

\begin{proof} The proof follows from Proposition \ref{equiv01} and Lemma \ref{entropia crescente}. 
\end{proof}

\begin{remark} It follows from the previous corollary that if $Dom(\sigma)=X$ and $X$ is compact, then the entropy of $\sigma$ introduced in Definition \ref{defofentropyviacovers} coincides with the usual topological entropy, as defined  in \cite[Definition 7.6]{Walters}, for example.
\end{remark}

We showed above that our definition of entropy via covers generalizes the usual definition of entropy via covers. Our next goal is to show that the entropy relative to a compact set of a Deaconu-Renault system can be computed using a sequence of covers with diameter converging to zero. We start recalling some usual concepts regarding covers in a metric space. 

\begin{definition}
Let $(X,d)$ be a metric space, $Y\subseteq X$ and $\beta$ a cover of $Y$. We define the diameter of $\beta$ by $diam(\beta)=sup\{diam(A):A\in \beta\},$ where $diam(A)=sup\{d(x,y):x,y\in A\}.$ If $Y$ is compact, then a Lebesgue number $\delta>0$ for $\beta$ is a number such that every subset of $Y$ with diameter less than $\delta$ is contained in some member of the cover $\beta$.
\end{definition}

The following lemma will be useful.

\begin{lemma}\label{Lemma04}Let $(X,\sigma)$ be a metric Deaconu-Renault system and let $K\subseteq X$ be a compact set.
\begin{enumerate}
    \item If $\alpha$ and $\beta$ are covers of $K$, with $\alpha$ a subcover of $\beta$, then $h(\beta,\sigma,K)\leq h(\alpha,\sigma,K)$.
    \item For every cover $\alpha$ of $K$,
    there exists $\delta>0$ such that if $\beta$ is a cover of $K$ with $diam(\beta)\leq \delta$, then $h(\alpha,\sigma,K)\leq h(\beta,\sigma,K)$. 
\end{enumerate}
\end{lemma}

\begin{proof} To prove the first item, let  $n\in \N$. Since $\alpha\subseteq\beta$, we have that $\alpha_n\subseteq \beta_n$ and hence each subcover $\gamma$ of $\alpha_n$ is also a subcover of $\beta_n$. Therefore, $N(\beta_n,K_n)\leq N(\alpha_n,K_n)$, from where the result follows.

We now prove the second item. Let $\alpha$ be an open cover of $K$ and define $\alpha^K=\{A\cap K: A\in \alpha\}$. Notice that $\alpha^K$ is an open cover (in $K$) of the compact space $K$. Let $\delta>0$ be a Lebesgue number of $\alpha^K$, and let $\beta$ be an open cover (in $X$) of $K$, such that $diam(\beta)<\delta$. Define $\beta^K=\{B\cap K:B\in \beta\}$. Fix $n\in \N$. We show that $N(\alpha_n,K_n)\leq N(\beta_n,K_n)$. Let $\gamma$ be a subcover of minimal cardinality of $\beta_n$, so that $N(\beta_n,K_n)$ equals the cardinality of $\gamma$. 

Recall that each $Y\in \gamma$ is of the form $Y=B_0\cap \sigma^{-1}(B_1)\cap\ldots \cap \sigma^{-n}(B_n)$, with $B_i\in \beta$ for each $i\in \{0,1\ldots ,n\}$. So, 
$$Y\cap K_n=(B_0\cap K)\cap \sigma^{-1}(B_1\cap K)\cap\ldots \cap \sigma^{-n}(B_n\cap K).$$

Since $diam(\beta)<\delta$, we have that $diam(\beta^K)<\delta$. Then, for each $C\in \beta^K$ there exists some $D\in \alpha^K$ such that $C\subseteq D$. Moreover, for each $i\in \{0,1.,..,n\}$, there exists $A_i\in \alpha$ such that $B_i\cap K \subseteq A_i\cap K$ (since  $B_i\cap K\in \beta^K$). Define
$Z=A_0\cap \sigma^{-1}(A_1)\cap \ldots \cap \sigma^{-n}(A_n)$, which is an element of $\alpha_n$. Then, $Y\cap K_n\subseteq Z\cap K_n.$

We have shown above that for each $Y\in \gamma$ there exists a set, which we denote by $Z_Y$, in $\alpha_n$ such that $Y\cap K_n\subseteq Z_Y\cap K_n$. Hence,  $$\bigcup\limits_{Y\in \gamma}(Y\cap K_n)\subseteq \bigcup\limits_{Y\in \gamma}(Z_Y\cap K_n).$$ 
Since $K_n\subseteq \bigcup\limits_{Y\in \gamma}Y$ we obtain that
$K_n=\bigcup\limits_{Y\in \gamma}(Y\cap K_n),$ and hence
$$K_n= \bigcup\limits_{Y\in \gamma}(Y\cap K_n)\subseteq \bigcup\limits_{Y\in \gamma}(Z_Y\cap K_n)\subseteq\bigcup\limits_{Y\in \gamma}Z_Y.$$
Therefore, $\{Z_Y:Y\in \gamma\}$ is a subcover (of the set $K_n$) of the cover $\alpha_n$. Since the cardinality of the set $\{Z_Y:Y\in \gamma\}$ is less or equal to the cardinality of $\gamma$, which is $N(\beta_n,K_n)$, it follows that $N(\alpha_n,K_n)\leq N(\beta_n,K_n)$. Consequently, $h(\alpha, \sigma, K)\leq h(\beta,\sigma, K)$.

\end{proof}

Next, we prove that the entropy of $\sigma$ relative to a compact set can be computed using a sequence of covers whose diameter converges to zero.

\begin{proposition}\label{diametro indo pra zero}
Let $(X,\sigma)$ be a metric Deaconu-Renault system, $K\subseteq X$ compact, and  $\{\alpha^n\}_{n=1}^{\infty}$ a sequence of open covers of $K$ with $diam(\alpha^n)\rightarrow 0$. Then, $h(\sigma,K)=\lim\limits_{n\rightarrow \infty} h( \alpha^n,\sigma,K)$.
\end{proposition}

\begin{proof}
Let $\{\alpha^n\}_{n=1}^\infty$ be a sequence of open covers of $K$ such that $diam(\alpha^n)\rightarrow 0$.

Suppose that $h(\sigma,K)=\infty$. Let $M>0$ and let $\alpha$ be an open cover of $K$ such that $h(\alpha,\sigma,K)\geq M$. From Lemma \ref{Lemma04}, there exists $\delta>0$ such that for each open cover $\beta$ of $K$, with $diam(\beta)<\delta$, the inequality $h(\beta,\sigma,K)\geq h(\alpha,\sigma,K)$ is true. Let $n_0\in \N$ be such that $diam(\alpha_n)<\delta$ for every $n\geq n_0$. Then, for $n\geq n_0$, we have that $h(\alpha^n,\sigma,K)\geq h(\alpha,\sigma,K)$. Therefore, $\lim\limits_{n\rightarrow \infty}h(\alpha^n,\sigma,K)=\infty$.  

To finish the proof, suppose that $h(\sigma,K)<\infty$ and fix an $\varepsilon>0$. From the definition of $h(\sigma, K)$, there exists an open cover $\alpha$ of $K$ such that $h(\sigma, K)-\varepsilon\leq h(\alpha,\sigma, K)\leq h(\sigma, K)$. Let $\delta>0$ be as in Lemma~\ref{Lemma04}, and let $n_1\in \N$ be such that $diam(\alpha^n)\leq\delta$ for every $n\geq n_1$. Then, by Lemma~\ref{Lemma04} and the definition of $h(\sigma, K)$, we have that $h(\sigma, K)-\varepsilon\leq h(\alpha^n,\sigma, K)\leq h(\sigma, K)$ for each $n\geq n_1$. This means that $\lim\limits_{n\rightarrow \infty}h(\alpha^n,\sigma,K)=h(\sigma, K)$, as desired.
\end{proof}

Our goal in the remainder of the section is to show that the topological entropy is a lower bound to the metric entropy of a Deaconu-Renault system. We set up the ground for this result below.

\begin{lemma}\label{lemma03}
Let $(X,\sigma)$ be a metric Deaconu-Renault system and let $K\subseteq X$ be compact. If $\alpha$ is an open cover of $K$ and $\delta$ is a Lebesgue number for $\alpha^K=\{A\cap K:A\in \alpha\}$, then
$$N(\alpha_n,K_n)\leq sspan(n+1,\frac{\delta}{8},\sigma,K)$$ for each $n\in \N$. 
\end{lemma}

\begin{proof} Let $n\in \N$ and $B$ be an $(n+1,\frac{\delta}{4},\sigma,K\cap Dom(\sigma^n))$-spanning set of smallest cardinality.
Then, for every $y\in K\cap Dom(\sigma^n)$, there exists $x\in B$ such that $d_{n+1}(x,y)\leq \frac{\delta}{4}<\frac{\delta}{2}$. 
and hence
$$K\cap Dom(\sigma^n)\subseteq\bigcup_{x\in B}\left(\bigcap_{i=0}^n \ \sigma^{-i}\left( \mathcal{B}\left(\sigma^i(x),\frac{\delta}{2}\right)  \right)\right).$$
Therefore,
\begin{eqnarray*}
K_n&=&(K\cap Dom(\sigma^n))\cap K_n \subseteq \left(\bigcup_{x\in B}\left(\bigcap_{i=0}^n  \sigma^{-i}\left( \mathcal{B}\left(\sigma^i(x),\frac{\delta}{2}\right)  \right)\right)\right)\bigcap K_n \\
&=&\bigcup_{x\in B}\left(\bigcap_{i=0}^n  \sigma^{-i}\left( \mathcal{B}\left(\sigma^i(x),\frac{\delta}{2}\right)  \right)\bigcap\limits_{i=1}^n \sigma^{-i}(K)\right)\subseteq
\bigcup_{x\in B}\left(\bigcap_{i=0}^n  \sigma^{-i}\left( \mathcal{B}\left(\sigma^i(x),\frac{\delta}{2}\right)  \bigcap K\right)\right).
\end{eqnarray*}

Now, since $\delta$ is a Lebesgue number for $\alpha^K$, for each $x\in B$ and $i\in \{0,\ldots ,n\}$ there exists some $A_i^x\in \alpha$ such that $$\mathcal{B}\left(\sigma^i(x), \frac{\delta}{2}\right)\bigcap K\subseteq A_i^x\cap K\subseteq A_i^x.$$ 

Therefore, $K_n\subseteq \bigcup\limits_{x\in B}\left(\bigcap\limits_{i=0}^n\sigma^{-i}(A_i^x)\right)$ and then, since $\bigcap\limits_{i=0}^n\sigma^{-i}(A_i^x)$ is an element of $\alpha_n$ for each $x\in B$, we obtain that 
$N(\alpha_n,K_n)$ is less or equal to the cardinality of $B$, which is equal to $span(n+1,\frac{\delta}{4}, \sigma, K\cap Dom(\sigma^n))$. Applying the second item of Proposition~\ref{propnovossepspan} we conclude that $$N(\alpha_n, K_n)\leq sspan(n+1,\frac{\delta}{8}, \sigma, K).$$
\end{proof}

We now have all the ingredients necessary to prove that the topological entropy of a Deaconu-Renault system is a lower bound to the metric entropy.

\begin{theorem}\label{zebra} Let $(X,\sigma)$ be a Deaconu-Renault system on the metric space $(X,d)$. Then, $h(\sigma)\leq h_d(\sigma)$.
\end{theorem}

\begin{proof}
Let $K\subseteq X$ be a compact set and let $(\alpha^m)_{m\in \N}$ be a sequence of open covers of $K$ such that $diam(\alpha^m)\rightarrow 0$. For each $m$, let $\delta_m$ be a Lebesgue number for the cover $\{A\cap K:A\in \alpha^m\}$. By Lemma~\ref{lemma03} (Item i.), we have that $N(\alpha_n^m, K_n)\leq sspan(n, \frac{\delta_m}{8},\sigma, K)$ for each $m,n\in\N$, and therefore $h(\alpha^m,\sigma, K)\leq h_{\frac{\delta_m}{8}}(\sigma, K, d)$ for each $m\in\N$.
Moreover, since $diam(\alpha^m)\rightarrow 0$, we have that $\delta_m\rightarrow 0$. Then, 
$$h(\sigma,K)=\lim\limits_{n\rightarrow \infty}h(\alpha^m,\sigma,K)\leq \lim\limits_{n\rightarrow \infty}h_{\frac{\delta_m}{8}}(\sigma, K, d)=h_d(\sigma, K).$$
Finally, taking the supremum over all the compact sets $K\subseteq X$ in the inequality above, we obtain that $h(\sigma)\leq h_d(\sigma)$.
\end{proof}

\section{Entropy of graphs and ultragraphs}\label{graphEntropy}

In this section, we focus on computing the entropy of Deaconu-Renault systems associated with graphs and ultragraphs. To contextualize our work with the existing literature, we make a study of the Deaconu-Renault system (and its entropy) associated with a row-finite graph. We also compute the topological entropy associated with the Renewal shift ultragraph, an example that models interesting dynamics and has been studied in the literature, for example in \cite{BEFR, MR2293630}.

Before we proceed, we recall the necessary concepts and notations regarding ultragraphs, which also apply to graphs.

\subsection{Preliminaries on ultragraphs}\label{section ultragraphs}

\begin{definition}\label{def of ultragraph}
An ultragraph is a quadruple $\mathcal{G}=(G^0, \mathcal{G}^1, r,s)$ consisting of two countable sets $G^0, \mathcal{G}^1$, a map $s:\mathcal{G}^1 \to G^0$, and a map $r:\mathcal{G}^1 \to P(G^0)\setminus \{\emptyset\}$, where $P(G^0)$ stands for the power set of $G^0$.
\end{definition}

\begin{definition}\label{def of mathcal{G}^0}
Let $\mathcal{G}$ be an ultragraph. Define $\mathcal{G}^0$ to be the smallest subset of $P(G^0)$ that contains $\{v\}$ for all $v\in G^0$, contains $r(e)$ for all $e\in \mathcal{G}^1$, and is closed under finite unions and nonempty finite intersections.
\end{definition}

Let $\mathcal{G}$ be an ultragraph. A finite path in $\mathcal{G}$ is either an element of $\mathcal{G}
^{0}$ or a sequence of edges $e_{1}\ldots  e_{k}$ in $\mathcal{G}^{1}$ where
$s\left(  e_{i+1}\right)  \in r\left(  e_{i}\right)  $ for $1\leq i\leq k$. If
we write $\alpha=e_{1}\ldots  e_{k}$, then the length $\left|  \alpha\right|  $ of
$\alpha$ is $k$. The length $|A|$ of a path $A\in\mathcal{G}^{0}$ is
zero. We define $r\left(  \alpha\right)  =r\left(  e_{k}\right)  $ and
$s\left(  \alpha\right)  =s\left(  e_{1}\right)  $. For $A\in\mathcal{G}^{0}$,
we set $r\left(  A\right)  =A=s\left(  A\right)  $. The set of
finite paths in $\mathcal{G}$ is denoted by $\mathcal{G}^{\ast}$. An infinite path in $\mathcal{G}$ is an infinite sequence of edges $\gamma=e_{1}e_{2}\ldots $ in $\prod \mathcal{G}^{1}$, where
$s\left(  e_{i+1}\right)  \in r\left(  e_{i}\right)  $ for all $i$. The set of
infinite paths in $\mathcal{G}$ is denoted by $\mathfrak
{p}^{\infty}$. The length $\left|  \gamma\right|  $ of $\gamma\in\mathfrak
{p}^{\infty}$ is defined to be $\infty$. A vertex $v$ in $\mathcal{G}$ is
called a sink if $\left|  s^{-1}\left(  v\right)  \right|  =0$ and is
called an infinite emitter if $\left|  s^{-1}\left(  v\right)  \right|
=\infty$. 

For $n\geq1,$ we define
$\mathfrak{p}^{n}:=\{\left(  \alpha,A\right)  :\alpha\in\mathcal{G}^{\ast
},\left\vert \alpha\right\vert =n,$ $A\in\mathcal{G}^{0},A\subseteq r\left(
\alpha\right)  \}$. We specify that $\left(  \alpha,A\right)  =(\beta,B)$ if
and only if $\alpha=\beta$ and $A=B$. We set $\mathfrak{p}^{0}:=\mathcal{G}^{0}$ and we let $\mathfrak{p}:=\coprod\limits_{n\geq0}\mathfrak{p}^{n}$. We embed the set of finite paths $\GG^*$ in $\mathfrak{p}$ by sending $\alpha$ to $(\alpha, r(\alpha))$. We
define the length of a pair $\left(  \alpha,A\right)  $, $\left\vert \left(
\alpha,A\right)  \right\vert $, to be the length of $\alpha$, $\left\vert
\alpha\right\vert $. We call $\mathfrak{p}$ the ultrapath space
associated with $\mathcal{G}$ and the elements of $\mathfrak{p}$ are called
ultrapaths (or just paths when the context is clear). Each $A\in\mathcal{G}^{0}$ is regarded as an ultrapath of length zero and can be identified with the pair $(A, A)$. We may extend the range map $r$ and the source map $s$ to
$\mathfrak{p}$ by the formulas, $r\left(  \left(  \alpha,A\right)  \right)
=A$, $s\left(  \left(  \alpha,A\right)  \right)  =s\left(  \alpha\right)
$ and $r\left(  A\right)  =s\left(  A\right)  =A$.

We concatenate elements in $\mathfrak{p}$ in the following way: If $x=(\alpha,A)$ and $y=(\beta,B)$, with $|x|\geq 1, |y|\geq 1$, then $x\cdot y$ is defined if and only if
$s(\beta)\in A$, and in this case, $x\cdot y:=(\alpha\beta,B)$. Also, we
specify that:
\begin{equation}
x\cdot y=\left\{
\begin{array}
[c]{ll}
x\cap y & \text{if }x,y\in\mathcal{G}^{0}\text{ and if }x\cap y\neq\emptyset\\
y & \text{if }x\in\mathcal{G}^{0}\text{, }\left|  y\right|  \geq1\text{, and
if } s\left(  y\right) \in x \\
x_{y} & \text{if }y\in\mathcal{G}^{0}\text{, }\left|  x\right|  \geq1\text{,
and if }r\left(  x\right)  \cap y\neq\emptyset
\end{array}
\right.  \label{specify}
\end{equation}
where, if $x=\left(  \alpha,A\right)  $, $\left|  \alpha\right|  \geq1$ and if
$y\in\mathcal{G}^{0}$, the expression $x_{y}$ is defined to be $\left(
\alpha,A\cap y\right)  $. Given $x,y\in\mathfrak{p}$, we say that $x$ has $y$ as an initial segment if
$x=y\cdot x^{\prime}$, for some $x^{\prime}\in\mathfrak{p}$, with $s\left(
x^{\prime}\right)  \in r\left(  y\right) $. 

We extend the source map $s$ to $\mathfrak
{p}^{\infty}$, by defining $s(\gamma)=s\left(  e_{1}\right)  $, where
$\gamma=e_{1}e_{2}\ldots $. We may concatenate pairs in $\mathfrak{p}$, with
infinite paths in $\mathfrak{p}^{\infty}$ as follows. If $y=\left(
\alpha,A\right)  \in\mathfrak{p}$, and if $\gamma=e_{1}e_{2}\ldots \in
\mathfrak{p}^{\infty}$ are such that $s\left(  \gamma\right)  \in r\left(
y\right)  =A$, then the expression $y\cdot\gamma$ is defined to be
$\alpha\gamma=\alpha e_{1}e_{2}\ldots \in\mathfrak{p}^{\infty}$. If $y=$
$A\in\mathcal{G}^{0}$, we define $y\cdot\gamma=A\cdot\gamma=\gamma$ whenever
$s\left(  \gamma\right)  \in A$. Of course $y\cdot\gamma$ is not defined if
$s\left(  \gamma\right)  \notin r\left(  y\right)  =A$.

Next, we recall the construction of the Deaconu-Renault system associated with an ultragraph, as in \cite{GRultrapartial, TascaDaniel}. We start with a couple of definitions.

\begin{definition}
\label{infinte emitter} For each subset $A$ of $G^{0}$, let
$\varepsilon\left(  A\right)  $ be the set $\{ e\in\mathcal{G}^{1}:s\left(
e\right)  \in A\}$. We say that a set $A$ in $\mathcal{G}^{0}$ is an
infinite emitter whenever $\varepsilon\left(  A\right)  $ is infinite. We say that $A$ is a minimal infinite emitter if it is an infinite emitter that contains no proper subsets (in $\GG^0$) that are infinite emitters. We denote the set of all minimal infinite emitters in $r(\alpha)$ by $M_\alpha$.
\end{definition}

Since the edge shift of a graph is usually considered for graphs without sinks, we restrict ourselves to ultragraphs without sinks. In this case, the Deaconu-Renault system associated with the ultragraph satisfies Hypothesis~\ref{niceone}. Moreover, we will only deal with ultragraphs that satisfy Condition~RFUM (see \cite{GRultrapartial}), that is, for which the range of each edge is a finite union of minimal infinite emitters union with a finite set of vertices (which are not infinite emitters). We make these assumptions explicit below:

{\bf Throughout assumption:} From now on, all ultragraphs (and, consequently, all graphs) in this paper are assumed to have no sinks and to satisfy Condition~(RFUM).

Associated with an ultragraph, we define the topological space $X= \mathfrak{p}^{\infty} \cup X_{fin}, $ where
 $$X_{fin} = \{(\alpha,A)\in \mathfrak{p}: |\alpha|\geq 1 \text{ and } A\in M_\alpha \}\cup
 \{(A,A)\in \GG^0: A \text{ is a minimal infinite emitter}\}, $$ and the topology has a basis given by the collection $$\{D_{(\beta,B)}: (\beta,B) \in \mathfrak{p}, |\beta|\geq 1\ \} \cup \{D_{(\beta, B),F}:(\beta, B) \in X_{fin}, F\subset \varepsilon\left( B \right), |F|<\infty \},$$ where for each $(\beta,B)\in \mathfrak{p}$ we have that $$D_{(\beta,B)}= \{(\beta, A): A\subset B \text{ and } A\in M_\beta \}\cup\{y \in X: y = \beta \gamma', s(\gamma')\in B\},$$ and, for $(\beta,B)\in X_{fin}$ and $F$ a finite subset of $\varepsilon\left( B \right)$,  $$D_{(\beta, B),F}=  \{(\beta, B)\}\cup\{y \in X: y = \beta \gamma', \gamma_1' \in \ \varepsilon\left( B \right)\setminus F\}.$$
 
\begin{remark}\label{cylindersets} For every $(\beta,B)\in \mathfrak{p}$, we identify $D_{(\beta, B)}$ with $D_{(\beta, B),F}$, where $F=\emptyset$. Furthermore, we call the basic elements of the topology of $X$ given above by cylinder sets.
\end{remark}

Next, we define the Deaconu-Renault system associated with an ultragraph. 

\begin{definition} 
 For $n\in\{0,1,2,\ldots \}$, define the following subsets of $X$:
$$X^n=\p^n\cap X_{fin}=\{x\in X_{fin}:|x|=n\}; 
\ X^{\geq n}=\displaystyle\bigcup_{k\geq n} X^k\text{ and }
 X_{\infty}^{\geq n}=X^{\geq n}\cup \p^\infty .$$
	\end{definition}

Notice that $X_{\infty}^{\geq 1}=X\backslash X^0$ and that $X_{\infty}^{\geq 0}=X$.
	
	\begin{definition}\label{def_shift}
	Let $\GG$ be an ultragraph and $X$ as above. We define the shift map $\sigma:X_{\infty}^{\geq 1}\rightarrow X$ by:
		
		$\sigma(x)=
		\left\{
		\begin{array}{ll}
		\gamma_2\gamma_3\ldots , & \text{if $x=\gamma_1\gamma_2\ldots  \in \p^\infty$};\\
		(\gamma_2\ldots \gamma_n,A), & \text{if $x=(\gamma_1\ldots \gamma_n,A) \in X^{\geq 2}$};\\
		(A,A), & \text{if $x=(\gamma_1,A) \in X^1$}.\\
				\end{array}
		\right.$  
		
		For $n> 1$ we define $\sigma^n$ as the composition $n$ times of $\sigma$, and for $n=0$ we define $\sigma^0$ as the identity. When we write $\sigma^n(x)$ we are implicitly assuming that $x\in X_{\infty}^{\geq n}$. 
		
	\end{definition}

It is shown in \cite[Proposition~5.4]{TascaDaniel} that $(X,\sigma)$ is a Deaconu-Renault system. We state it below for completeness. 

\begin{proposition}\label{DRsystem}\cite[Proposition~5.4]{TascaDaniel}
Let $\mathcal G$ be an ultragraph. Then $(X,\sigma)$ is a Deaconu-Renault system.
\end{proposition}

In the next subsection, we recall a family of metrics defined in $X$ and prove results about them, which are useful in the computation of the metric entropy of the Deaconu-Renault system associated with an ultragraph and are interesting on their own.

\subsection{A family of metrics in the ultragraph shift space}\label{familyofmetrics}

Based on ideas in \cite{OTW} and \cite{Webster}, an explicit family of metrics for the space $X$ built in the previous sections was described in \cite{BrunoDaniel}. We recall it below.

To define the metric, first we list the elements of $\mathfrak{p}$ as $\mathfrak{p} = \{ p_1, p_2, p_3, \ldots  \}$. Then, for each $x,y \in X$, we define
\begin{equation}\label{definmetricanova}
d_X (x,y) := \begin{cases} 1/2^i & \text{$i \in \N$ is the smallest value such that $p_i$ is an initial} \\  & \text{ \ \ \ segment of one of $x$ or $y$ but not the other,} \\
0 & \text{if $x=y$}.
\end{cases}
\end{equation} 

\begin{remark}\label{ordem} The metric $d_X$ induces the topology on $X$ described in the previous section (see \cite{BrunoDaniel}).
Moreover, $d_X$ depends on the order we choose for $\mathfrak{p} = \{ p_1, p_2, p_3, \ldots  \}$. Nevertheless, we show below that the metrics induced by different orderings of $\p$ are uniformly equivalent.
\end{remark}

\begin{proposition}\label{sono}
Let $\GG$ be an ultragraph and $d_i$, $i=1,2$, be metrics on $X$ obtained from possibly distinct enumerations of $\mathfrak{p}$. Then, $d_1$ and $d_2$ are uniformly equivalent.
\end{proposition}

\begin{proof}
Let $\{p_0',p_1',p_2',\ldots \}$ and $\{p_0'',p_1'',p_2'',\ldots \}$ be two enumerations of $\p$ and let $d_1$ and $d_2$ be the associated metrics, respectively. 
We prove that $id:(X,d_1) \rightarrow (X,d_2)$ is uniformly continuous. 
Let $\varepsilon>0$ and choose $J$ such that $\frac{1}{2^J}<\varepsilon$. Let $k>0$ be such that, for all $0\leq i \leq J$, there exists $l_i<k$ such that $p_{l_i}'=p_i''$. Now, notice that if $d_1(x,y)<\frac{1}{2^k}$ then $d_2(x,y)<\frac{1}{2^J}<\varepsilon$, and so $id$ is uniformly continuous as desired. The proof that $id^{-1}$ is uniformly continuous is analogous.
\end{proof}

The proposition above motivates the definition of the metric entropy of an ultragraph.

\begin{definition}\label{janeiro}
Let $\mathcal G$ be an ultragraph and $(X,\sigma)$ be the associated Deaconu-Renault system. By the metric entropy of $\mathcal G$ we mean the metric entropy of the system $(X,\sigma)$, with $X$ equipped with the metric $d_X$ described above.
\end{definition}

For the next result, we need to recall the following lemma.

\begin{lemma}\label{range}(\cite[Lemma 2.10]{GilDan}) Let $\GG$ be an ultragraph with no sinks that satisfies Condition~(RFUM). Then, each $A\in \GG^0$ can be written uniquely as $A = \displaystyle \bigcup_{n=1}^k A_n,$ where there exists an unique $k$ such that $|A_k|<\infty$ and $\varepsilon(A_k)< \infty$, $A_j$ is a minimal infinite emitter for $j\neq k$, and $A_j \cap A_k = \emptyset$ for all $j\neq k$. 
\end{lemma}

Next, we prove that the metric $d_X$ is uniformly equivalent to a metric in which it is not necessary to list all elements of $\mathfrak p$. 
\begin{proposition}\label{mais metricas equivalentes}
Let $\GG$ be an ultragraph that satisfies Condition~(RFUM) and $d_X$ be the metric on $X$ obtained from an enumeration of $\mathfrak{p}$, say $\{p_0',p_1',p_2',\ldots \}$. Let $\{p_0'',p_1'',p_2'',\ldots \}$ be an enumeration of the subset of ultrapaths $$S=\{(\alpha,A): (\alpha,A)\in \p, \text{ and } A \text{ is either a vertex or a minimal infinite emitter}\},$$ and let $d_1$ be the associated metric (which is defined analogously to $d_X$, see Equation~\ref{definmetricanova}). Then, $d_X$ and $d_1$ are uniformly equivalent.
\end{proposition}

\begin{proof}
The proof that $d_1$ is uniformly continuous with respect to $d_X$ is analogous to the proof of Proposition~\ref{sono}. We show that $d_X$ is uniformly continuous with respect to $d_1$. 
Given $\varepsilon>0$, let $J>0$ be such that $\frac{1}{2^J}<\varepsilon$. 

Use Lemma~\ref{range} and write each $r(p_i')$, $0\leq i \leq J$ as $r(p_i')= \cup A_{n_i}^i$, with $A_{k_i}^i$ the distinguished set with $\varepsilon(A_{k_i}^i)<\infty$. 
Let $\mathcal D $ be the set whose elements are the sets $A_{n_i}^i$, with $n_i\neq k_i$, and the unitary sets formed by vertices $v$ in $A_{k_i}^i$, $0\leq i \leq J$. Clearly, $\mathcal D$ is finite.  Write each $p_i'$ as $p_i'=(\alpha_i,A_i)$ and let \[k=\max\{j: p_j''=(\alpha_i,A) \text{ for some A}\in \mathcal D, 1\leq i \leq J\}.\]

Then, $d_1(x,y)<\frac{1}{2^k}$ implies that $d_X(x,y)<\frac{1}{2^J}<\varepsilon$. Indeed, suppose that $d_1(x,y)<\frac{1}{2^k}$ and $d_X(x,y)\geq\varepsilon$. In this case, there exists $p_i'=(\alpha_i,A)$, $i< J$ that is an initial segment of $x$ and not from $y$ or vice-versa. Assume, without loss of generality, that $p_i'$ is an initial segment of $x$ and not from $y$. Decompose $A$ as in Lemma~\ref{range}, say, $A=\cup_{l=1}^n A_l$. Then, there must be one $A_{l_0}$ such that $(\alpha_i,A_{l_0})$ is an initial segment of $x$ but not of $y$ (notice that if $|x|=0$, since $x\in X$, then $x$ is a minimal infinite emitter. Therefore, $x\cap A_l $ is finite for all $l$, and hence there exists $l_0$ 
such that $A_{l_0}\cap x = x$). If $A_{l_0}$ is an infinite emitter, then we are done. If not, then $A_{l_0}$ is a finite set with $\varepsilon (A_{l_0})<\infty$ and there must be a vertex $v$ in $A_{l_0}$ such that $(\alpha_i,v)$ is an initial segment of $x$ but not of $y$. In both cases, we conclude that $d_1(x,y)\geq \frac{1}{2^k}$, a contradiction. Hence, $d_X$ is uniformly continuous with respect to $d_1$ and the proof is finished.
\end{proof}

We finish this subsection by observing that the metric entropy can be computed using any of the metrics presented in this subsection. 

\begin{corollary}\label{enum_met}
Let $\GG$ be an ultragraph. Then the metric entropy $\mathcal G$ is independent of the choice of enumeration of $\p$. Moreover, the metric entropy can be computed using the metric $d_1$ defined above.
\end{corollary}

\begin{proof}
This follows from Propositions \ref{sono} and \ref{invariance by the metric}.
\end{proof}

\begin{remark}\label{enumeracao tabajara} Following Proposition~\ref{mais metricas equivalentes}, the metric entropy of $\mathcal G$ is independent of the enumeration of the set $S$ (defined in Proposition~\ref{mais metricas equivalentes}). So, to compute the metric entropy we can fix the most adequate enumeration of S. Next we fix an enumeration, which will be used later. Let $\mathcal{G}$ be a countable ultragraph, with $\mathcal{G}^1=\{e_i:i\in \N\}$ and enumerate the set $ \{A\in \mathcal G^0: A \text{ is a vertex or an minimal infinite emitter}\}$ as $\{B_i:i\in \N\}$. To begin, list all the elements $(x,A)$, where $x$ belongs to the set of paths of length less or equal to $1$ generated by $\{B_1,e_1\}$; next, list all the elements $(x,A)$, where $x$ belongs to the set of paths of length less or equal to $2$ generated by $\{B_1, B_2,e_1,e_2\}$ without repeating elements previously listed; now, list all the elements $(x,A)$, where $x$ belongs to the set of paths of length less or equal to $3$ generated by $\{B_1,B_2,B_3, e_1,e_2,e_3\}$ without repeating elements previously listed; continue the procedure successively. 
\end{remark}

\subsection{The entropy of a row-finite graph.}

In this section, we focus on row-finite graphs and compute the entropy of the associated Deaconu-Renault system $(X,\sigma)$, connecting our results with the existing literature. Recall that a graph can be seen as an ultragraph such that the range map takes values on the set of vertices (not in the power set of vertices). More commonly, a graph is a quadruple $(E^0, E^1,r,s)$, where $E^0$ stands for the set of vertices, $E^1$ stands for the set of edges, and $r,s$ are the range and source maps, respectively. The notation introduced in Section~\ref{section ultragraphs} applies to graphs and we make the addition of the following definitions.

\begin{definition}
We say that a graph $(E^0, E^1,r,s)$ is row-finite if $|s^{-1}(v)|<\infty$ for all $v\in E^0$. We say that a row-finite graph is locally finite if $|r^{-1}(v)|<\infty$ for all $v\in E^0$.
\end{definition}

Notice that, in the row-finite case, $X$ is the usual edge shift space associated with a graph (that is, $X$ consists of all infinite paths equipped with the usual cylinder set topology) and $\sigma$ is the one-sided shift map. This allows us to connect our entropy theory with the existing literature. The first result we need is the following.

\begin{proposition}
\label{dorsal} Let $E$ be a row-finite graph. 
Then, the metric defined in Equation~(\ref{definmetricanova}) is uniformly continuous with respect to the usual metric associated with an edge shift space, that is, the metric given by $d(x,y)=\frac{1}{2^i}$, where $i$ is the smallest index such that $x_i\neq y_i$. If the graph is finite, then the metrics are uniformly equivalent.
\end{proposition}
\begin{proof}

Let $E$ be a row-finite graph, $d_X$ be the metric of Equation~\ref{definmetricanova}, and $d$ as in the statement of the proposition. We show that $d_X$ is uniformly continuous with respect to $d$, i.e., $Id_X:(X,d)\rightarrow (X,d_X)$ is uniformly continuous. 
List the elements of $\p$ as $\{p_1,p_2,\ldots \}$. Given an $\varepsilon>0$, choose $k\in \N$ such that $\frac{1}{2^k}<\varepsilon$ and let $m=\displaystyle\max_{1\leq i\leq k}|p_i|$. Define $\delta=\frac{1}{2^{m+1}}$. Notice that if $d(x,y)<\delta$, then $x_i=y_i$ for all $i\in\{1,\ldots , m\}$, and so any $j$ such that $p_j$ is an initial segment of one of $x$ or $y$ but not the other must satisfy $|p_j|\geq m+1$. Hence, $d_X(x,y)<\frac{1}{2^k}<\varepsilon$. 

Now, suppose that the graph is finite. We show that $Id_X:(X,d_X)\rightarrow (X,d)$ is uniformly continuous. Given $\varepsilon>0$, choose $k\in\N$ such that $\frac{1}{2^k}<\varepsilon$. Consider the set $A=\{p_i:|p_i|\leq k\}$, which is finite. Let $j=\max\{i:p_i\in A\}$ and define $\delta=\frac{1}{2^{j+1}}$. Then, $d_X(x,y)<\delta$ implies that $d(x,y) < \varepsilon$ as desired. 

\end{proof}

\begin{remark}
Notice that the condition that the graph is finite in the proposition above is sharp. Indeed, even for locally finite graphs the final statement of the proposition above does not hold. For example, for the graph with a countable number of vertices $\{v_i\}_{i\in \N}$, and edges $\{e_i,f_i\}_{i\in \N}$ such that $s(e_i)=s(f_i)=v_i$ and $r(e_i)=r(f_i)=v_{i+1}$ we have that $Id_X:(X,d_X)\rightarrow (X,d)$ is not uniformly continuous.
 \begin{center}
\setlength{\unitlength}{2mm}
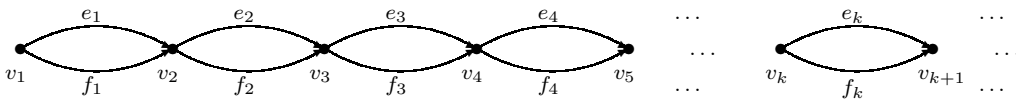
	\begin{picture}(70,10)
		\put(1,3){\scriptsize$v_1$}
		\put(2,5){\circle*{0.7}}
		\put(11,3){\scriptsize$v_2$}
		\put(12,5){\circle*{0.7}}
		\put(21,3){\scriptsize$v_3$}
		\put(22,5){\circle*{0.7}}
		\put(31,3){\scriptsize$v_4$}
		\put(32,5){\circle*{0.7}}
		\put(41,3){\scriptsize$v_5$}
		\put(42,5){\circle*{0.7}}

	    \qbezier(2,5)(7,8)(12,5)
		\put(12,5){\vector(2,-1){0}}
		\qbezier(12,5)(17,8)(22,5)
		\put(22,5){\vector(2,-1){0}}	
		\qbezier(22,5)(27,8)(32,5)
		\put(32,5){\vector(2,-1){0}}
        \qbezier(32,5)(37,8)(42,5)
    	\put(42,5){\vector(2,-1){0}}
	    
		\qbezier(2,5)(7,2)(12,5)
		\put(12,5){\vector(2,1){0}}
		\qbezier(12,5)(17,2)(22,5)
		\put(22,5){\vector(2,1){0}}	
		\qbezier(22,5)(27,2)(32,5)
		\put(32,5){\vector(2,1){0}}
        \qbezier(32,5)(37,2)(42,5)
    	\put(42,5){\vector(2,1){0}}
          
        \put(46,4.6){\scriptsize$\ldots $}
	    \put(45,7){\scriptsize$\ldots $}
        \put(45,2.2){\scriptsize$\ldots $}

        \put(6,7){\scriptsize$e_1$}
		\put(16,7){\scriptsize$e_2$}
		\put(26,7){\scriptsize$e_3$}
		\put(36,7){\scriptsize$e_4$}
		
		\put(6,2.2){\scriptsize$f_1$}
		\put(16,2.2){\scriptsize$f_2$}
		\put(26,2.2){\scriptsize$f_3$}
		\put(36,2.2){\scriptsize$f_4$}
		
		\put(51,3){\scriptsize$v_{k}$}
		\put(52,5){\circle*{0.7}}
		\put(61,3){\scriptsize$v_{k+1}$}
		\put(62,5){\circle*{0.7}}
		
		\put(56,7){\scriptsize$e_k$}
		\qbezier(52,5)(57,8)(62,5)
		\put(62,5){\vector(2,-1){0}}
		
		\put(56,2){\scriptsize$f_k$}
		\qbezier(52,5)(57,2)(62,5)
		\put(62,5){\vector(2,1){0}}
		
		\put(66,4.6){\scriptsize$\ldots $}
	    \put(65,7){\scriptsize$\ldots $}
        \put(65,2.2){\scriptsize$\ldots $}
\end{picture}
\end{center}
 To see that this is true, let $\{p_1,p_2,\ldots \}$ be an enumeration of $\p$ such that: for each $k\in \N$ there exists $i_0\in \N$ such that $e_k=p_{i_0}$, $f_k=p_{i_0+1}$, and the edges $e_k$ and $f_k$ do not appear in any $p_i$ with $i<i_0$. 
 Now, let $(i_k)_{k\in \N}$ be the subsequence such that $p_{i_k}=e_k$ for each $k\in \N$ and, for each $k\in \N$, let $x^k,y^k$ be infinite paths such that the first edge of $x^k$ is $e_k$ and the first edge of $y^k$ is $f_k$. Then, we have that $d_X(x^k,y^k)=\frac{1}{2^{i_k}}$, which converges to zero as $k$ increases, and $d(x^k,y^k)=\frac{1}{2}$ for each $k\in \N$.

\end{remark}

For finite graphs, using Proposition~\ref{dorsal}, we obtain that the metric entropy of the graph (as defined in Definition~\ref{janeiro}) can be computed classically via the metric entropy definition. It can also be computed via the topological definition of Section~\ref{topoent}. We make this precise below (recall that $\mathfrak{p}^n$ denotes the paths of length $n$ in the graph).

\begin{proposition}\label{entropy of finite graphs}
Let $E$ be a finite graph and $(X,\sigma)$ be the associated Deaconu-Renault system. Then, $$h_d(\sigma) = \lim_{n\rightarrow \infty} \frac{1}{n} log |\mathfrak p^n|.$$ 
\end{proposition}
\begin{proof}

If $E$ is finite then the associated shift space is compact and $Dom(\sigma^n)=X$ for all $n$. So, the result follows from Proposition~\ref{invariance by the metric} and \cite[Theorem~7.13]{Walters}.

\end{proof}

In the classical literature of countable edge shift spaces, the Gurevic entropy of a locally finite graph $E$, $h_G(E)$, is computed as the supremum of the entropies of finite, connected subgraphs, that is, \[h_G(E)= \sup_{H\subseteq E, \ H \text{ finite}} h(X_H,\sigma),\]  see \cite[Prop.~7.2.6]{Kitchens} or \cite[Section~13.9]{LindMarcus}. This entropy can be obtained as the Bowen metric entropy, see \cite{gurevich, Pet2, Sal1}, associated with the following metric.

Given a row-finite graph $E$, enumerate the set of edges as $E^1=\{e_i:i\in \N\}$ and let $h:E^1\rightarrow \N$ be the bijection defined by $h(e_i)=i$.  Following \cite{gurevich}, the map $d:X\times X \rightarrow \R$ defined by \[d(x,y)=\sum\limits_{i=1}^\infty\frac{|\frac{1}{h(x_i)}-\frac{1}{h(y_i)}|}{2^i}\] is a metric. We show next that, for locally-finite graphs, the metric $d$ above is uniformly equivalent to the metric $d_X$.

\begin{proposition}\label{jacare}
Let $E$ be a row-finite graph and $d_X$ be the metric given in Equation~\ref{definmetricanova}. Then, $d_X$ is uniformly continuous with respect to the metric $d$ defined above. If $E$ is locally finite, then $d_X$ and $d$ are uniformly equivalent.
\end{proposition}
\begin{proof}

First, we show $d_X$ is uniformly continuous with respect to $d$.
Let $p_1, p_2, \ldots $ be an enumeration of $\p$. 
Let $\varepsilon>0$ and consider $k\in \N$ such that $0<\frac{1}{2^k}<\varepsilon$. 
For each $p_i$, with $i=1,\ldots ,k$, we can write $p_i=e_i^1\ldots  e_i^{|p_i|}$ where $e_i^j\in E^1$. Let $$M=\max\{h(e_i^j):i=1,\ldots ,k; j=1,\ldots , |p_i|\},$$ define $L=\displaystyle\max_{1\leq i\leq k} |p_i|$, let $\delta'=\frac{1}{M+1}-\frac{1}{M+2}$ and $\delta=\dfrac{\delta'}{2^L}$. 

Now, fix $x=(x_i)_{i\in \N}, y=(y_i)_{i\in \N}\in X$, such that $d(x,y)<\delta$. Then
$$\sum_{i=1}^L \frac{|\frac{1}{h(x_i)}-\frac{1}{h(y_i)}|}{2^i}\leq \sum_{i=1}^\infty \frac{|\frac{1}{h(x_i)}-\frac{1}{h(y_i)}|}{2^i}=d(x,y)<\delta=\frac{\delta'}{2^L}.$$
So, for each $i\in \{1,\ldots ,L\}$, we have $$\left|\frac{1}{h(x_i)}-\frac{1}{h(y_i)}\right|\leq\sum_{i=1}^L \left|\frac{1}{h(x_i)}-\frac{1}{h(y_i)}\right|\leq\sum_{i=1}^L 2^{L-i}\left|\frac{1}{h(x_i)}-\frac{1}{h(y_i)}\right|< \delta'.$$  
Next we show that for a fixed $i\in \{1,\ldots ,L\}$ it holds that, or $x_i=y_i$, or $h(x_i)>M$ and $h(y_i)>M$.

Suppose first that $h(x_i)\leq M$ and $h(y_i)\leq M$. Supposing $h(x_i)\neq h(y_i)$, we get  $|\frac{1}{h(x_i)}-\frac{1}{h(y_i)}|\geq\frac{1}{M}-\frac{1}{M+1}$, which is impossible since  $|\frac{1}{h(x_i)}-\frac{1}{h(y_i)}|< \delta'=\frac{1}{M+1}-\frac{1}{M+2}$. Therefore, $x_i=y_i$ in this case. 

Suppose now that $h(x_i)\leq M$ and $h(y_i)\geq M+1$. In this case, $\frac{1}{h(x_i)}\geq \frac{1}{M}$ and $\frac{1}{h(y_i)} \leq \frac{1}{M+1}$, and so $|\frac{1}{h(x_i)}-\frac{1}{h(y_i)}|\geq \frac{1}{M}-\frac{1}{M+1}$, which is also impossible since $|\frac{1}{h(x_i)}-\frac{1}{h(y_i)}|<\delta'=\frac{1}{M+1}-\frac{1}{M+2}$. Of course, it is also impossible $h(y_i)\leq M$ and $h(x_i)\geq M+1$.

So, it follows that, for $i\in \{1,\ldots ,L\}$, or $y_i=x_i$, or $h(x_i) >M$ and $h(y_i)>M$.

Let $j\in \N$ be the smallest value such that $p_j$ is an initial segment of one of $x$ or $y$ but not the other. Then $d_X(x,y)=\frac{1}{2^j}$. Write $p_j=f_1f_2\ldots f_{|p_j|}$ where $f_i\in E^1$ for each $i$. If $|p_j|\leq L$ then $x_i\neq y_i$ for some $i\in \{1,\ldots ,L\}$ and therefore $h(x_i)> M$ and $h(y_i)>M$. In particular, $h(f_i)>M$. Therefore, $f_i$ is not an edge of any  element of $\{p_1,\ldots ,p_k\}$, and so $p_j\notin \{p_1,\ldots ,p_k\}$. Therefore $j>k$. If $|p_j|>L$ then obviously $p_j\notin\{p_1,\ldots ,p_k\}$ and so also $j>k$. Then it follows that $d_X(x,y)=\frac{1}{2^j}< \frac{1}{2^k}<\varepsilon$.

Now, we assume the graph is locally finite and we show that $Id_X:(X,d_X)\to (X,d)$ is uniformly continuous. Let $\varepsilon>0$. 
Fix $L\in \N$ such that for all $i,j>L$ it holds that $\left|\frac{1}{i}-\frac{1}{j}\right|<\frac{\varepsilon}{2}$,
 and let $k\in \N$ such that $\sum\limits_{i=k}^\infty \frac{1}{2^i}<\frac{\varepsilon}{2}$.

Let $$A:=\{p_i \in \p; |p_i|\leq k \text{ and }  p_i \text{ contains at least one of the edges }\{e_1,e_2,\ldots ,e_L\} \text{ in its composition}\}.$$ Notice that since the graph is locally finite then $A$ is a finite set. Now let $M=\max \{i: p_i\in A\}$ and take $\delta=\frac{1}{2^M}$. 

Fix two elements $x=(x_i)_{i\in \N}, y=(y_i)_{i\in \N}\in X$ with $d_X(x,y)<\delta$. We show that $d(x,y)<\varepsilon$. Since $d_X(x,y)<\delta=\frac{1}{2^k}$ then $d_X(x,y)=\frac{1}{2^j}$, where $j$ is the smallest index such that $p_j$ is the initial path of one of $x$ or $y$ but not of the other. 

Suppose that $x_i\in \{e_1,\ldots e_L\}$ for some $i\in \{1,\ldots ,k\}$. In this case, $x_1x_2\ldots x_k\in A$ and supposing $x_1x_2\ldots x_k\neq y_1y_2\ldots y_k$, then $d_X(x,y)\geq \frac{1}{2^M}=\delta$, which is impossible. Therefore, $x_1x_2\ldots x_k=y_1y_2..y_k$ in this case. The same holds if $y_i\in \{e_1,\ldots e_L\}$ for some $i\in \{1,\ldots ,k\}$. Notice that since $x_1\ldots x_k=y_1\ldots y_k$, then $$d(x,y)=\sum\limits_{i=k+1}^\infty\frac{|\frac{1}{h(x_i)}-\frac{1}{h(y_i)}|}{2^i}<\sum\limits_{i=k+1}^\infty\frac{1}{2^i}<\frac{\varepsilon}{2}<\varepsilon.$$

Case $x_i,y_i\notin \{e_1,\ldots ,e_L\}$ for each $i\in \{1,\ldots ,k\}$ then $h(x_i)>L$ and $h(y_i)>L$ for each $i\in \{1,\ldots ,k\}$, and then, since $|\frac{1}{h(x_i)}-\frac{1}{h(i_i)}|<\frac{\varepsilon}{2}$  it holds that

$$d(x,y)=\sum\limits_{i=1}^\infty \frac{|\frac{1}{h(x_i)}-\frac{1}{h(y_i)}|}{2^i}=\sum\limits_{i=1}^k \frac{|\frac{1}{h(x_i)}-\frac{1}{h(y_i)}|}{2^i}+\sum\limits_{i=k+1}^\infty \frac{|\frac{1}{h(x_i)}-\frac{1}{h(y_i)}|}{2^i}<$$
$$<\sum\limits_{i=1}^k \frac{\frac{\varepsilon}{2}}{2^i}+\sum\limits_{i=k+1}^\infty \frac{1}{2^i}<\frac{\varepsilon}{2}+\frac{\varepsilon}{2}=\varepsilon.$$

\end{proof}

\begin{corollary}
Let $E$ be a locally finite graph and $(X,\sigma)$ the associated Deaconu-Renault system. Then, $h_d(\sigma)=h_G(E)$. 
\end{corollary}
\begin{proof}

This follows from the proposition above and Proposition~\ref{invariance by the metric}.

\end{proof}

We show below that the topological entropy of the Deaconu-Renault system associated with any row-finite graph can be computed as the supremum over the entropies of the Deaconu-Renault systems associated with finite, connected subgraphs. 

\begin{theorem}\label{sup dos subgrafos}
Let $E$ be a row-finite graph and $(X,\sigma)$ the associated Deaconu-Renault system. Then, 
\[h(\sigma)= \sup_{H\subseteq E} h(\sigma_H),\] where the supremum is taken over all finite subgraphs $H$ of $E$, and $(X_H,\sigma_H)$ is the Deaconu-Renault system associated with the graph $H$.
\end{theorem}
\begin{proof}

Let $H$ be a finite subgraph of $E$. Define $X_H$ as the shift space associated with $H$, that is, $X_H$ consists of all infinite paths in $H$ with the topology given by the usual cylinder sets. This space is compact. Let $m\in \N$ and let $F$ be the set of all the paths in $H$ with length $m$, which is a finite set. For each $\beta\in F$, denote by $D_{\beta}$ the cylinder set in $X_H$ consisting of all paths in $H$ with initial segment $\beta$, and notice that $\alpha=\{D_\beta:\beta\in F\}$ is an open cover of $X_H$ (in $X_H$).

For each $\beta\in F$ let $C_\beta$ be the cylinder set in $X$ determined by $\beta$. Define $K=\bigcup\limits_{\beta\in F}C_\beta$, which is a compact set in $X$, and notice that $\gamma=\{C_\beta:\beta\in F\}$ is an open cover of $K$ in $X$.

Now we divide the proof into a few claims. The first one below follows from direct calculations and its proof is left to the reader. 

{\it Claim 1: For a given set $\{\beta_0,\beta_1,\ldots ,\beta_n\}\subseteq F$ we have that either  
$$D_{\beta_0}\cap \sigma_{H}^{-1}(D_{\beta_1})\cap\ldots \cap \sigma_H^{-n}(D_{\beta_n})=\emptyset=C_{\beta_0}\cap \sigma^{-1}(C_{\beta_1})\cap\ldots \cap \sigma^{-n}(C_{\beta_n}),$$
 or there exists a path $\xi$ in $H$, with $|\xi|=m+n$, such that
$$D_{\beta_0}\cap \sigma_{H}^{-1}(D_{\beta_1})\cap\ldots \cap \sigma_H^{-n}(D_{\beta_n})=D_{\xi} \text{ and } C_{\beta_0}\cap \sigma^{-1}(C_{\beta_1})\cap\ldots \cap \sigma^{-n}(C_{\beta_n})=C_{\xi}.$$ }

{\it Claim 2: Fix $n\in \N$. Then, $N(\gamma_n,K_n)= N(\alpha_n, (X_H)_n)$.}

First we show that $N(\gamma_n,K_n)\leq N(\alpha_n, (X_H)_n)$. Notice that each element of $A\in\alpha_n$ is of the form $$A=D_{\beta_0}\cap \sigma_H^{-1}(D_{\beta_1})\cap\ldots \cap \sigma_H^{-n}(D_{\beta_n})$$ and, moreover, $(X_H)_n$ is a union of such sets.

Let $\{A_1,..,A_p\}$ be a subcover of $\alpha_n$ of minimal cardinality covering $(X_H)_n$, which implies that $N(\alpha_n,(X_H)_n)=p$.
Since $A_i \in \alpha_n$ and $A_i$ is non-empty we have, by Claim~1, that there exists a path $\xi_i\in H$ with $|\xi_i|=m+n$ and $A_i=D_{\xi_i}$, for each $1\leq i\leq p$. This means that $(X_H)_n\subseteq D_{\xi_1}\cup\ldots \cup D_{\xi_p}$. 

We show that $K_n\subseteq C_{\xi_1}\cup\ldots \cup C_{\xi_p}$. Notice that $K_n$ is a union of the elements of $\gamma_n$, and each non-empty element of $\gamma_n$ is of the form 
$$B=C_{a_0}\cap \sigma^{-1}(C_{a_1})\cap \ldots  \cap \sigma^{-n}(C_{a_n}).$$
Since $D_{a_0}\cap \sigma_H^{-1}(D_{a_1})\cap \ldots  \cap \sigma_H^{-n}(D_{a_n})$ is an element of $\alpha_n$ (and hence a subset of $(X_H)_n\subseteq D_{\xi_1}\cup\ldots \cup D_{\xi_n}$) and $D_{a_0}\cap \sigma_H^{-1}(D_{a_1})\cap \ldots  \cap \sigma_H^{-n}(D_{a_n})=D_{\xi}$ for some path $\xi$ in $H$ with $|\xi|=m+n$ (by Claim~1), we have that $D_{a_0}\cap \sigma_H^{-1}(D_{a_1})\cap \ldots  \cap \sigma_H^{-n}(D_{a_n})=D_{\xi_j}$ for some $1\leq j\leq p$. Again, by Claim~1, we have that $C_{a_0}\cap \sigma^{-1}(C_{a_1})\cap \ldots  \cap \sigma^{-n}(C_{a_n})=C_{\xi_j}$. Therefore, $B$ is a subset of $ C_{\xi_1}\cup\ldots \cup C_{\xi_p}$ and hence $K_n\subseteq C_{\xi_1}\cup\ldots \cup C_{\xi_p}$. 

It remains to show that each $C_{\xi_i}$, with $1\leq i\leq p$, is an element of $\gamma_n$. Recall that $D_{\xi_i}=A_i\in \alpha_n$. Hence, for each $1\leq i\leq p$, there exists $\{b_0,b_1,\ldots ,b_n\}\subseteq F$ such that $D_{b_0}\cap \sigma_H^{-1}(D_{b_1})\cap \ldots  \cap \sigma_H^{-n}(D_{b_n})=D_{\xi_i}$. By Claim 1, 
$C_{b_0}\cap \sigma^{-1}(C_{b_1})\cap \ldots  \cap \sigma^{-n}(C_{b_n})=C_{\xi_i}$ and therefore $C_{\xi_i}\in \gamma_n$ as desired.

We conclude that $\gamma_n$ has a sub-cover with $p$ elements, which covers $K_n$. Hence, $N(\gamma_n,K_n)\leq p=N(\sigma_n,(X_H)_n)$.

Similarly, one shows that $N(\sigma_n,(X_H)_n)\leq N(\gamma_n, K_n)$ and this finishes the proof of  Claim~2.

\vspace{0.5pc}
{\it Claim 3: Let $H$ and $K$ as above. Then, $h(\sigma_H, X_H)=h(\sigma, K)$}

Let $\alpha$ and $\gamma$ be as above. By Claim 2, $N(\alpha_n, (X_H)_n)=N(\gamma_n, K_n)$ for each $n\in \N$. Hence, $h(\gamma, \sigma_H, X_H)=h(\gamma, \sigma, K)$. Recall that $\alpha$ and $\gamma$ depend on the fixed $m$ at the beginning of the proof of this theorem. It is not hard to see that, with the metric fixed in Remark~\ref{enumeracao tabajara}, we have that $diam(\alpha)\rightarrow 0$ and $diam(\gamma)\rightarrow 0$ as $m\rightarrow \infty$. Then, by Proposition~\ref{diametro indo pra zero}, we get that $h(\sigma, K)=h(\sigma_H, X_H)$ and Claim 3 is proved.

\vspace{0.5pc}

Now we prove the theorem. 

Since $h(\sigma, K)\leq h(\sigma)$ for each compact set $K\subseteq X$ we have, by Claim 3, that $h(\sigma_H)=h(\sigma_H, X_H)\leq h(\sigma)$, for each finite subgraph $H$ of $E$. Therefore, 
$$\sup\limits_Hh(\sigma_H)\leq h(\sigma),$$ where the supremum is taken over all the finite subgraphs $H$ of $E$.

It remains to show that $h(\sigma)\leq \sup\limits_Hh(\sigma_H)$. For this, we show that for each compact set $L\subseteq X$, there exists some finite subgraph $H$ of $E$ such that $h(\sigma, L)\leq h(\sigma_H)$.

Let $L$ be a compact subset of $X$ and $m\in \N$. Since $L$ is compact, we have that $L$ is contained in a finite union of cylinder sets $C_a$ in $X$, with $a\in F$. Since $E$ has no sinks and is row-finite we can assume that $|a|=m$ for each $a$. Let $H$ be the finite subgraph of $E$ generated by all the edges of all the paths $a$ in $F$. For this subgraph $H$, let $X_H$, $K$, $\alpha$, and $\gamma$ be as at the beginning of the proof of this theorem, and notice that $L\subseteq K$. Then, since $L\subseteq K$, we have by Lemma~\ref{entropia crescente} that   $h(\sigma, L)\leq h(\sigma, K)$ and, by Claim 3, we have that $h(\sigma, K)=h(\sigma_H)$. 
Therefore, $h(\sigma, L)\leq h(\sigma_H)$. Taking supremum over $L$ and $H$, we get that $h(\sigma)\leq \sup\limits_Hh(\sigma_H)$ as desired. 
\end{proof}

\begin{corollary}\label{andorinha}
Let $E$ be a locally finite graph. Then the metric entropy of $E$ and the topological entropy of $E$ coincide.
\end{corollary}
\begin{proof}
By Proposition~\ref{jacare} the metric entropy coincides with Gurevich's entropy, which in turn can be computed as the supremum of entropies of finite graphs. For finite graphs, it is well known that the metric and topological entropies agree. The result now follows from the theorem above.
\end{proof}

Next, we compute entropy for a few examples. 

\begin{example}\label{ex2}
Let $E=(E^0, E^1,r,s)$ be the locally finite graph given in the picture below:

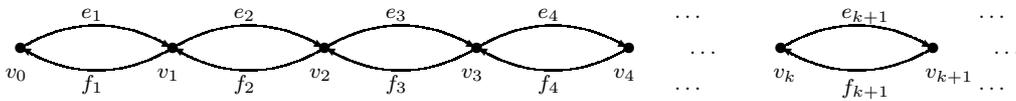
\begin{center}
\setlength{\unitlength}{2mm}
	\begin{picture}(70,10)
		\put(1,3){\scriptsize$v_0$}
		\put(2,5){\circle*{0.7}}
		\put(11,3){\scriptsize$v_{1}$}
		\put(12,5){\circle*{0.7}}
		\put(21,3){\scriptsize$v_{2}$}
		\put(22,5){\circle*{0.7}}
		\put(31,3){\scriptsize$v_{3}$}
		\put(32,5){\circle*{0.7}}
		\put(41,3){\scriptsize$v_{4}$}
		\put(42,5){\circle*{0.7}}

	    \qbezier(2,5)(7,8)(12,5)
		\put(12,5){\vector(2,-1){0}}
		\qbezier(12,5)(17,8)(22,5)
		\put(22,5){\vector(2,-1){0}}	
		\qbezier(22,5)(27,8)(32,5)
        \put(32,5){\vector(2,-1){0}}
    	\qbezier(32,5)(37,8)(42,5)
    	\put(42,5){\vector(2,-1){0}}
	    
		\qbezier(2,5)(7,2)(12,5)
		\put(2,5){\vector(-2,1){0}}
		\qbezier(12,5)(17,2)(22,5)
		\put(12,5){\vector(-2,1){0}}	
		\qbezier(22,5)(27,2)(32,5)
		\put(22,5){\vector(-2,1){0}}
        \qbezier(32,5)(37,2)(42,5)
    	\put(32,5){\vector(-2,1){0}}
          
        \put(46,4.6){\scriptsize$\ldots $}
	    \put(45,7){\scriptsize$\ldots $}
        \put(45,2.2){\scriptsize$\ldots $}

        \put(6,7){\scriptsize$e_1$}
		\put(16,7){\scriptsize$e_2$}
		\put(26,7){\scriptsize$e_3$}
		\put(36,7){\scriptsize$e_4$}
		
		\put(6,2.2){\scriptsize$f_1$}
		\put(16,2.2){\scriptsize$f_2$}
		\put(26,2.2){\scriptsize$f_3$}
		\put(36,2.2){\scriptsize$f_4$}
		
		\put(51.5,3){\scriptsize$v_{k}$}
		\put(52,5){\circle*{0.7}}
		\put(61.5,3){\scriptsize$v_{k+1}$}
		\put(62,5){\circle*{0.7}}
		
		\put(56,7){\scriptsize$e_{k+1}$}
		\qbezier(52,5)(57,8)(62,5)
		\put(62,5){\vector(2,-1){0}}
		
		\put(56,2){\scriptsize$f_{k+1}$}
		\qbezier(52,5)(57,2)(62,5)
		\put(52,5){\vector(-2,1){0}}
		
		\put(66,4.6){\scriptsize$\ldots $}
	    \put(65,7){\scriptsize$\ldots $}
        \put(65,2.2){\scriptsize$\ldots $}
\end{picture}
\end{center}

Let $(X,\sigma)$ be the associated Deaconu-Renault system. For each finite subgraph $H$ of $E$ denote by $(X_H,\sigma_H)$ the Deaconu-Renault system associated with the graph $H$, where $X_H$ is the set of all the infinite paths in $H$. 
Notice that each finite subgraph $G$ of $E$ is contained in a graph of the form $H_m=(E_m^0, E_m^1, r, s)$, where $E_m^0=\{v_0, v_1,\ldots ,v_m\}$, and $E_m^1=\{e_1, f_1, e_2, f_2,\ldots ,e_m, f_m\}$. So, by Theorem~\ref{sup dos subgrafos}, the entropy of $(X,\sigma)$ is the supremum of the entropies $h(\sigma_{H_m})$ of $(X_{H_m}, \sigma_{H_m})$. We compute $h(\sigma_{H_m})$ for each $m\in \N$. Fix $m\in \N$, $n\in \N$, and  denote by $\mathfrak{p}^n$ the set of all paths of length $n$ in $H_m$. Notice that $|\mathfrak{p}^n|\leq 2m 2^n$, since $H_m$ contains $2m$ edges and each vertex of $H_m$ emits 2 edges. Moreover, let $a=f_1e_1$ and $b=f_me_m$. Then, the range of each element in the set $\{a, e_2, f_2, e_3, f_3,\ldots , e_{m-1}, f_{m-1}, b\}$ emits two edges and hence the cardinality of the subset of all the elements of $\mathfrak{p}^n$ generated by the set $\{a, e_2, f_2, e_3, f_3,\ldots , e_{m-1}, f_{m-1}, b\}$ is $(2m-2)2^{n-2}$. Therefore, $(2m-2)2^{n-2}\leq |\mathfrak{p}^n|\leq 2m2^n$ for each $n\in \N$. Hence, 
$\lim\limits_{n\rightarrow \infty}\frac{1}{n}log|\mathfrak{p^n}|=log(2)$ for each $m\in \N$ and from Proposition~\ref{entropy of finite graphs}, we obtain that $h(\sigma_{H_m})=log(2)$ for each $m\in \N$. It follows, by Theorem~\ref{sup dos subgrafos}, that $h(\sigma)=log(2)$.
\end{example}

\begin{example}
Let $E=(E^0,E^1,r,s)$ be the disconnected graph consisting of the above graph and a rose of three petals, see the picture below. 

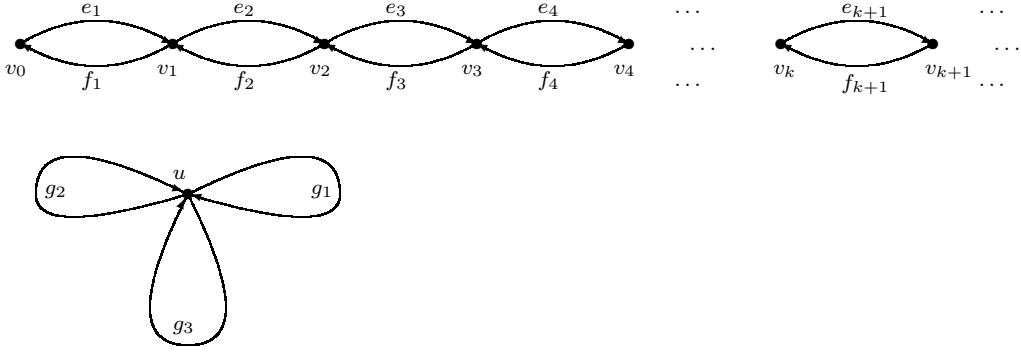
\begin{center}
\setlength{\unitlength}{2mm}
	\begin{picture}(70,10)
		\put(1,3){\scriptsize$v_0$}
		\put(2,5){\circle*{0.7}}
		\put(11,3){\scriptsize$v_{1}$}
		\put(12,5){\circle*{0.7}}
		\put(21,3){\scriptsize$v_{2}$}
		\put(22,5){\circle*{0.7}}
		\put(31,3){\scriptsize$v_{3}$}
		\put(32,5){\circle*{0.7}}
		\put(41,3){\scriptsize$v_{4}$}
		\put(42,5){\circle*{0.7}}

	    \qbezier(2,5)(7,8)(12,5)
		\put(12,5){\vector(2,-1){0}}
		\qbezier(12,5)(17,8)(22,5)
		\put(22,5){\vector(2,-1){0}}	
		\qbezier(22,5)(27,8)(32,5)
        \put(32,5){\vector(2,-1){0}}
    	\qbezier(32,5)(37,8)(42,5)
    	\put(42,5){\vector(2,-1){0}}
	    
		\qbezier(2,5)(7,2)(12,5)
		\put(2,5){\vector(-2,1){0}}
		\qbezier(12,5)(17,2)(22,5)
		\put(12,5){\vector(-2,1){0}}	
		\qbezier(22,5)(27,2)(32,5)
		\put(22,5){\vector(-2,1){0}}
        \qbezier(32,5)(37,2)(42,5)
    	\put(32,5){\vector(-2,1){0}}
          
        \put(46,4.6){\scriptsize$\ldots $}
	    \put(45,7){\scriptsize$\ldots $}
        \put(45,2.2){\scriptsize$\ldots $}

        \put(6,7){\scriptsize$e_1$}
		\put(16,7){\scriptsize$e_2$}
		\put(26,7){\scriptsize$e_3$}
		\put(36,7){\scriptsize$e_4$}
		
		\put(6,2.2){\scriptsize$f_1$}
		\put(16,2.2){\scriptsize$f_2$}
		\put(26,2.2){\scriptsize$f_3$}
		\put(36,2.2){\scriptsize$f_4$}
		
		\put(51.5,3){\scriptsize$v_{k}$}
		\put(52,5){\circle*{0.7}}
		\put(61.5,3){\scriptsize$v_{k+1}$}
		\put(62,5){\circle*{0.7}}
		
		\put(56,7){\scriptsize$e_{k+1}$}
		\qbezier(52,5)(57,8)(62,5)
		\put(62,5){\vector(2,-1){0}}
		
		\put(56,2){\scriptsize$f_{k+1}$}
		\qbezier(52,5)(57,2)(62,5)
		\put(52,5){\vector(-2,1){0}}
		
		\put(66,4.6){\scriptsize$\ldots $}
	    \put(65,7){\scriptsize$\ldots $}
        \put(65,2.2){\scriptsize$\ldots $}
        
        \put(12,-4){\scriptsize$u$}
        \put(13,-5){\circle*{0,7}}
        \qbezier(13,-5)(23,0)(23,-5)
        \qbezier(13,-5)(23,-8)(23,-5)
        \put(13.2,-5){\vector(-3,1){0}}
        \put(21.1,-5){\scriptsize$g_1$}
        
		\qbezier(13,-5)(3,0)(3,-5)
        \qbezier(13,-5)(3,-8)(3,-5)
        \put(12.7,-4.8){\vector(3,-1){0}}
        \put(3.6,-5){\scriptsize$g_2$}
        
        \qbezier(13,-5)(8,-15)(13,-15)
        \qbezier(13,-5)(18,-15)(13,-15)
        \put(12.8,-5.3){\vector(1,2){0}}
        \put(12,-14){\scriptsize$g_3$}
        
\end{picture}
\end{center}

\vspace{3 cm}
In Example~\ref{ex2} we computed the entropy of the top subgraph, which is log(2). The entropy of the subgraph generated by $\{g_1,g_2,g_3\}$ and $\{u\}$ is log(3). Then, by Proposition~\ref{prop:entropy-restriction}, the metric entropy of the graph $E$ is log(3) and by Corollary~\ref{andorinha} the topological entropy is the same as the metric entropy.

\end{example}

\subsection{The entropy of the Renewal shift}

An important example of a topologically mixing  countable Markov shift is the Renewal shift, which has infinite adjacency matrix given by \[A_\mathcal{G}=\begin{bmatrix} 
1 & 1 & 1 & 1 & \ldots \\
1 & 0 & 0 & 0 & \ldots \\
0 & 1 & 0 & 0 & \ldots \\
0 & 0 & 1 & 0 & \ldots \\
\vdots & \vdots & \vdots & \vdots & \ddots
\end{bmatrix}.\]
Dynamical properties and the (Gurevich) topological entropy of the associated Markov shift space (the usual shift space equipped with the product topology, see \cite{DDM}) are studied in \cite{MNTY, Sarig}.  In connection with C*-algebra theory and thermodynamic formalism, in \cite{BEFR}, a certain space is associated with the matrix above. This space is conjugate to the shift space (defined in \cite{GRultrapartial}) associated with the ultragraph $\mathcal G$ arising from $A_{\mathcal G}$ (see \cite{ClaramuntDaniel}), which in turn is the space we used to define the Deaconu-Renault system associated with $\mathcal G$. We recall that the ultragraph $\mathcal G$ associated with $A_\mathcal G$ has vertices $\{v_i:i\in \N\}$ and edges $\{e,f_i:i\in \N\}$, where  $s(e)=v_1$, $r(e)=\{v_i:i\in \N\}$, $s(f_i)=v_{i+1}$, and $r(f_i)=v_i$, for all $i\in \N$.

  \begin{center}
\setlength{\unitlength}{2mm}
	\begin{picture}(30,10)
		\put(1,-1.5){\scriptsize$v_1$}
		\put(2,0){\circle*{0.7}}
		\put(11,-1.5){\scriptsize$v_2$}
		\put(12,0){\circle*{0.7}}
		\put(21,-1.5){\scriptsize$v_3$}
		\put(22,0){\circle*{0.7}}
		\put(31,-1.5){\scriptsize$v_4$}
		\put(32,0){\circle*{0.7}}
		\put(41,-1.5){\scriptsize$v_5$}
		\put(42,0){\circle*{0.7}}
    	\put(45,0){\scriptsize$\ldots $}
			
		\qbezier(2,0)(-7,11)(2,10)    
		\put(2,0){\vector(-1,-2){0}}
		\qbezier(2,10)(7,8)(2,0)
		\put(12,0){\vector(3,-4){0}}	
		\qbezier(2,10)(4,10)(12,0)
		\put(22,0){\vector(2,-1){0}}	
		\qbezier(2,10)(4,10)(22,0)
		\put(32,0){\vector(3,-1){0}}
		\qbezier(2,10)(4,10)(32,0)
    	\put(42,0){\vector(4,-1){0}}
		\qbezier(2,10)(4,10)(42,0)
		\put(30,5){\scriptsize$\ldots $}
		\put(-2,10){\scriptsize$e$}
		
		\put(12,0){\vector(-1,0){10}}	
		\put(22,0){\vector(-1,0){10}}	
		\put(32,0){\vector(-1,0){10}}	
		\put(42,0){\vector(-1,0){10}}	
	
		\put(6,-1.2){\scriptsize$f_1$}
		\put(16,-1.2){\scriptsize$f_2$}
		\put(26,-1.2){\scriptsize$f_3$}
		\put(36,-1.2){\scriptsize$f_4$}

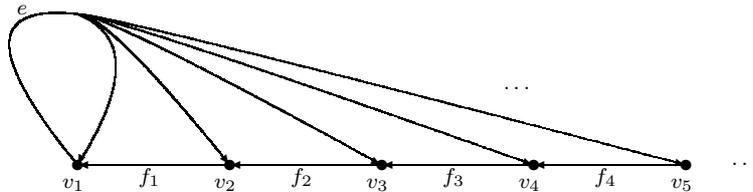
			\put(6,-4.2){\normalsize The Renewal shift ultragraph}

\end{picture}
\end{center}
\vspace{1cm}
Next, we compute the topological entropy of the Deaconu-Renault system associated with $\mathcal G$. For this, we construct a sequence of covers with diameter tending to zero (so we can use Proposition~\ref{diametro indo pra zero}). To compute the diameter of a cover we need to fix a metric, which is done once we choose an enumeration of the set of ultrapaths $\p$ or an enumeration of the set $S$ described in Proposition~\ref{mais metricas equivalentes}.
So, choose the enumeration fixed in Remark~\ref{enumeracao tabajara}.

Let $d$ be the metric induced by this enumeration and let $(X,d)$ be the Deaconu-Renault system associated with $\mathcal G$. It follows from Remark~\ref{ordem} and \cite[Proposition 3.12]{GRultrapartial} that $(X,d)$ is compact. Next, we define a cover of the space $X$.

Fix an $m\in \N$ and let $F=\{e,f_1,f_2,\ldots ,f_m\}$. Define the sets $$Q=\{\beta: \beta \text{ is a finite path in }\mathcal{G} \text{ with edges in } F, \text{ whose last edge is }  e \text{ and }|\beta|<m\}$$
and $$R=\{\beta: \beta \text{ is a finite path in }\mathcal{G} \text{ with edges in }F \text{ and }|\beta|=m\}.$$

Define an open cover of $X$ by $$\alpha^m=\left\{D_{(r(e),r(e)),F}\right\}\bigcup\{D_{\beta,F}:\beta\in Q\}\bigcup\{D_\gamma:\gamma \in R\}.$$

 To simplify notation, in what follows we denote $\alpha^m$ by $\alpha$. Notice that $\alpha$ is a finite cover and that all the cylinder sets of this cover are pairwise disjoint sets. 

Next, we compute $h(\alpha, \sigma, X)$, that is, the entropy of $\sigma$ relative to the open cover $\alpha$ (see Definition~\ref{defofentropyviacovers}). 
For this, we first observe that the inverse image by $\sigma$ of each cylinder set of $\alpha$ is a union of two cylinder sets.

To see that this holds for $D_{(r(e), r(e)),F}$, write
$$D_{(r(e), r(e)),F}=\bigcup\limits_{i=m+1}^\infty D_{f_i}.$$ Then, $$\sigma^{-1}(D_{(r(e),r(e)),F})=\sigma^{-1}\left(\bigcup\limits_{i=m+1}^\infty D_{f_i}\right)=\left(\bigcup\limits_{i=m+2}^\infty D_{f_i}\right)\cup\left(\bigcup\limits_{i=m+1}^\infty D_{ef_i}\right)=D_{(r(e),r(e)),F\cup\{f_{m+1}\}}\cup D_{e,F}.$$
To check the observation for cylinders of the form $D_{\beta,F}$ notice that, for each vertex $v$, there are exactly two edges containing $v$ in its range, one of them being $e$. Then, for a given $\beta \in Q$, there are exactly two edges containing $s(\beta)$ in its range, the edge $e$ and one more edge, which we call $a_\beta$. Therefore, $\sigma^{-1}(D_{\beta,F})$ is a union of two cylinder sets, that is, $$\sigma^{-1}(D_{\beta,F})=D_{e\beta,F}\cup D_{a_\beta\beta,F}.$$ Similarly, for each $\gamma \in R$, $$\sigma^{-1}(D_\gamma)=D_{e\gamma}\cup D_{a_\gamma\gamma}.$$ Therefore, 
$$\sigma^{-1}(\alpha)=\left\{D_{(r(e),r(e)),F\cup \{f_{m+1}\}}\cup D_{e,F}\right\}\cup \{D_{e\beta,F}\cup D_{a_\beta\beta,F}:\beta \in Q\}\cup \{D_{e\gamma} \cup D_{a_\gamma\gamma}:\gamma \in R\} .$$

The reader can now verify that $\alpha\vee \sigma^{-1}(\alpha)$ is the set of all the cylinder sets appearing as the elements of $\sigma^{-1}(\alpha)$, or, more specifically, $ \alpha_1=\alpha\vee \sigma^{-1}(\alpha)$ is equal to $$\left\{D_{(r(e),r(e)),F\cup \{f_{m+1}\}}\right\}\cup \left\{ D_{e,F} \right\}\cup \left\{D_{e\beta,F}:\beta \in Q\right\}\cup \left\{D_{a_\beta \beta,F}\right\}\cup \left\{D_{e\gamma}:\gamma \in R\right\}\cup \left\{D_{a_\gamma\gamma}:\gamma \in R\right\}.$$ 
Notice that the above collection has $2M$ cylinder sets, where $M$ is the cardinality of $\alpha$. Since all these cylinder sets are pairwise disjoint, we obtain that $$N(\alpha_1, X_1)=2M.$$

Proceeding similarly to what is done above, we conclude that the inverse image by $\sigma$ of each cylinder set appearing in $\sigma^{-1}(\alpha)$ is a union of two cylinder sets. Moreover, the cover $\alpha_2=\alpha\vee \sigma^{-1}(\alpha)\vee \sigma^{-2}(\alpha)$ of $X_2=X\cap \sigma^{-1}(X)\cap \sigma^{-2}(X)$ is the collection of all the cylinder sets appearing as the elements of  $\sigma^{-1}(\sigma^{-1}(\alpha))$. So,  $N(\alpha_2,X_2)=2^2M$. 

Proceeding inductively, we obtain that $N(\alpha_n,X_n)=2^nM$ for each $n\in \N$. Then, $$h(\alpha, \sigma, X)=\lim\limits_{n\rightarrow \infty}\frac{1}{n}H(\alpha_n,X_n)=\lim\limits_{n\rightarrow \infty}\frac{1}{n}log(2^nM)=log(2).$$

To finish, we go back to the original notation for the cover $\alpha$, that is, $\alpha^m$ (since this cover depends on $m$). It follows from the chosen enumeration of $S$, and from the definition of $d$, that $\lim\limits_{m\rightarrow \infty}diam(\alpha^m)=0$. Then, from Proposition~\ref{diametro indo pra zero}, we obtain that $h(\sigma,X)=\lim\limits_{m\rightarrow\infty}h(\alpha^m,\sigma, X)$ and therefore $h(\sigma, X)=log(2)$. Hence, from Lemma \ref{entropia crescente} and Definition \ref{defofentropyviacovers} we get that $h(\sigma)=log(2)$.

\newpage 
(Daniel Gon\c{c}alves, Danilo Royer and Felipe Augusto Tasca) 
 { Departamento de Matem\'atica, UNIVERSIDADE FEDERAL DE SANTA CATARINA, 88040-970, Florian\'opolis SC, Brazil.}
 
{\textit{Email Address:}}\texttt{
\href{mailto:daemig@gmail.com}{daemig@gmail.com}, \href{mailto:danilo.royer@ufsc.br}{danilo.royer@ufsc.br}, \href{mailto:tasca.felipe@gmail.com}{tasca.felipe@gmail.com}.}

\end{document}